\documentclass[12pt]{amsart}
\usepackage{amsmath}
\usepackage{amssymb}
\usepackage[all]{xy}
\usepackage{longtable}
\usepackage{eucal}

\makeatletter
\let\@wraptoccontribs\wraptoccontribs
\makeatother

\DeclareFontFamily{OT1}{rsfs}{}
\DeclareFontShape{OT1}{rsfs}{n}{it}{<-> rsfs10}{}
\DeclareMathAlphabet{\mathscr}{OT1}{rsfs}{n}{it}

\newtheorem{theorem}[equation]{Theorem}
\newtheorem{lemma}[equation]{Lemma}
\newtheorem{proposition}[equation]{Proposition}
\newtheorem{corollary}[equation]{Corollary}

\theoremstyle{definition}

\newtheorem{example}[equation]{Example}

\theoremstyle{remark}
\newtheorem{remark}[equation]{Remark}

\numberwithin{equation}{subsection}

\DeclareMathAlphabet{\matheur}{U}{eur}{m}{n}

\newcommand{\FF}{\mathbb{F}}
\newcommand{\ZZ}{\mathbb{Z}}
\newcommand{\QQ}{\mathbb{Q}}

\newcommand{\LL}{\mathbb{L}}
\newcommand{\TT}{\mathbb{T}}
\newcommand{\GG}{\mathbb{G}}

\newcommand{\CC}{\mathbb{C}}

\newcommand{\ba}{\mathbf{a}}

\newcommand{\bff}{\mathbf{f}}
\newcommand{\bg}{\mathbf{g}}
\newcommand{\bh}{\mathbf{h}}
\newcommand{\bm}{\mathbf{m}}
\newcommand{\bn}{\mathbf{n}}

\newcommand{\bx}{\mathbf{x}}

\newcommand{\bv}{\mathbf{v}}

\newcommand{\cR}{\mathcal{R}}
\newcommand{\cT}{\mathcal{T}}

\newcommand{\eA}{\matheur{A}}

\newcommand{\ek}{\matheur{k}}
\newcommand{\eK}{\matheur{K}}
\newcommand{\eM}{\matheur{M}}

\newcommand{\eR}{\matheur{R}}
\newcommand{\eX}{\matheur{X}}

\newcommand{\rB}{\mathrm{B}}
\newcommand{\rP}{\mathrm{P}}

\DeclareMathOperator{\Aut}{Aut} \DeclareMathOperator{\Lie}{Lie}
\DeclareMathOperator{\Ker}{Ker} \DeclareMathOperator{\GL}{GL}
\DeclareMathOperator{\Mat}{Mat} \DeclareMathOperator{\Cent}{Cent}
\DeclareMathOperator{\End}{End} 
\DeclareMathOperator{\Span}{Span} \DeclareMathOperator{\Gal}{Gal}
\DeclareMathOperator{\Hom}{Hom} \DeclareMathOperator{\Res}{Res}
\DeclareMathOperator{\DR}{DR} 
\DeclareMathOperator{\Id}{Id} 
\DeclareMathOperator{\trdeg}{tr.deg} 
\DeclareMathOperator{\Ext}{Ext}

\newcommand{\ok}{\overline{k}}

\newcommand{\oL}{\overline{L}}

\newcommand{\sep}{\mathrm{sep}}
\newcommand{\tr}{\mathrm{tr}}

\newcommand{\tPsi}{\widetilde{\Psi}}

\newcommand{\iso}{\stackrel{\sim}{\to}}
\newcommand{\Ga}{{\GG_{\mathrm{a}}}}

\newcommand{\one}{\mathbf{1}}
\newcommand{\Qbar}{\overline{\QQ}}

\newcommand{\power}[2]{{#1 [\![ #2 ]\!]}}
\newcommand{\laurent}[2]{{#1 (\!( #2 )\!)}}
\newcommand{\Rep}[2]{\mathbf{Rep}(#1,#2)}

\begin{document}

\title[Algebraic independence of periods and logarithms]{Algebraic
independence of periods and logarithms of Drinfeld modules}

\author{Chieh-Yu Chang}
\address{Department of Mathematics, National Tsing Hua University and
National Center for Theoretical Sciences, Hsinchu City 30042, Taiwan
  R.O.C.}

\email{cychang@math.cts.nthu.edu.tw}

\author{Matthew A. Papanikolas}
\address{Department of Mathematics, Texas A{\&}M University, College
Station, TX 77843, USA} \email{map@math.tamu.edu}

\contrib[with an appendix by]{Brian Conrad}
\address{Department of Mathematics, Stanford University, Stanford, CA 94305, USA}
\email{conrad@math.stanford.edu}

\thanks{B.~Conrad was supported by NSF grant DMS-0917686.  C.-Y. Chang was supported by an NCTS
postdoctoral fellowship. M.~A.~Papanikolas was supported by NSF
Grant DMS-0903838.}


\date{May 22, 2011}

\begin{abstract}
Let $\rho$ be a Drinfeld $A$-module with generic characteristic
defined over an algebraic function field.  We prove that all of the
algebraic relations among periods, quasi-periods, and logarithms of algebraic
points on $\rho$ are those coming from linear relations induced by
endomorphisms of $\rho$.
\end{abstract}


\maketitle
\section{Introduction}\label{Sec:Introduction}

\subsection{Drinfeld logarithms} In this paper we prove algebraic
independence results about periods, quasi-periods, logarithms, and
quasi-logarithms on Drinfeld modules, which are inspired by
conjectures from the theory of elliptic curves and abelian
varieties. More specifically, let $E$ be an elliptic curve defined
over the algebraic numbers $\Qbar$ with Weierstrass $\wp$-function
$\wp(z)$.  If $E$ does not have complex multiplication, then one
expects that the periods $\omega_1$, $\omega_2$ and quasi-periods
$\eta_1$, $\eta_2$ of $E$ are algebraically independent over
$\Qbar$; if $E$ has CM, one knows by a theorem of Chudnovsky that
their transcendence degree is $2$ over $\Qbar$, which also aligns
with expectations. Furthermore, if $u_1, \dots, u_n$ are complex
numbers with $\wp(u_i) \in \Qbar \cup \{\infty\}$ that are linearly
independent over $\End(E)$, then one expects that $u_1, \dots, u_n$
are algebraically independent over $\Qbar$.  Now the $\Qbar$-linear
relations among the elliptic logarithms or general abelian
logarithms of algebraic points have been studied extensively in the
past several decades, but algebraic independence of these logarithms
is still wide open (see \cite{BW,Waldschmidt}).  In the present
paper we prove complete analogues of these conjectures for Drinfeld
modules of arbitrary rank over general base rings.

Let $\FF_{q}$ be a finite field of $q$ elements in characteristic
$p$. Let $\eX$ be a smooth, projective, geometrically connected
curve over $\FF_{q}$. We fix a closed point $\infty$ on $\eX$, and
denote by $A$ the ring of functions on $\eX$ regular away from
$\infty$. Let $k$ be the function field of $\eX$ over $\FF_{q}$, and
let $k_{\infty}$ be the completion of $k$ at $\infty$. Let
$\CC_{\infty}$ be the completion of a fixed algebraic closure of
$k_{\infty}$ at $\infty$, and let $\ok$ be the algebraic closure of
$k$ in $\CC_{\infty}$.

In Drinfeld's \cite{Drinfeld} seminal work on elliptic modules, one
lets $\Lambda\subseteq \CC_{\infty}$ be an $A$-lattice of projective
rank $r$, i.e., a finitely generated discrete $A$-module in
$\CC_{\infty}$, and then forms the \emph{exponential function}
$e_{\Lambda}(z):=z\prod_{0\neq\lambda\in \Lambda}\left( 1-z/\lambda
\right)$. The Drinfeld $A$-module $\rho^{\Lambda}$ of rank $r$
associated to $\Lambda$ is the $\FF_{q}$-linear ring homomorphism
from $A$ into the endomorphism ring of the additive group
$\Ga/\CC_{\infty}$ that satisfies
$e_{\Lambda}(az)=\rho_{a}^{\Lambda}(e_{\Lambda}(z))$ for all $a\in
A$.  Thus $\rho^\Lambda$ induces an $A$-module structure on
$\CC_\infty$, which is isomorphic to $\CC_\infty/\Lambda$ via
$e_\Lambda(z)$.

We fix a rank $r$ Drinfeld $A$-module $\rho=\rho^{\Lambda}$ that is
defined over $\ok$; i.e., as a polynomial in the Frobenius
endomorphism, the coefficients of each $\rho_a$ are in $\ok$.  Let
$\exp_\rho(z) := e_{\Lambda}(z)$.  We set $\End(\rho):=\left\{x\in
\CC_{\infty} \mid x\Lambda\subseteq \Lambda \right\}$, which can be
identified with the endomorphism ring of $\rho$, and let $K_\rho$ be
its fraction field.

In analogy with the classical results of Schneider and Siegel on the
transcendence of elliptic logarithms of algebraic points, Yu
\cite{Yu86} proved the transcendence of nonzero elements $u\in
\CC_{\infty}$ with $\exp_{\rho}(u)\in\ok$.  Yu \cite{Yu97}
established further in this context the full analogue of Baker's
celebrated theorem on the linear independence of logarithms of
algebraic numbers.  Our first main result is stated as follows.

\begin{theorem}\label{T:Thm1Introd}
Let $\rho$ be a Drinfeld $A$-module defined over $\ok$. Let
$u_{1},\ldots,u_{n}\in \CC_{\infty}$ satisfy
$\exp_{\rho}(u_{i})\in\ok $ for $i=1,\ldots,n$. If
$u_{1},\ldots,u_{n}$ are linearly independent over $K_\rho$, then
they are algebraically independent over $\ok$.
\end{theorem}

Previous work of the authors \cite{CP,Papanikolas} has established
special cases of this theorem when $A$ is a polynomial ring and,
either $\rho$ has rank $1$, i.e., the Carlitz module case, or $\rho$
has rank $2$ without CM and the characteristic is odd.

\subsection{The period matrix and quasi-periodic functions} Based on
the analogy between Drinfeld modules and elliptic curves, Anderson,
Deligne, Gekeler, and Yu developed a theory of quasi-periodic
functions for $\rho$ and defined the \emph{de Rham group}
$H_{\DR}(\rho)$ (see \cite{Gekeler89, Yu90}). The group
$H_{\DR}(\rho)$ is an $r$-dimensional vector space over $\CC_\infty$
that parametrizes extensions of $\rho$ by $\Ga$.  Associated to each
$\delta \in H_{\DR}(\rho)$ there is the entire \emph{quasi-periodic
function} $F_\delta(z)$, which sets up the de Rham isomorphism (see
\cite{Gekeler89}),
\begin{equation}\label{E:deRhamIsomIntro}
\delta\mapsto (\omega\mapsto F_{\delta}(\omega)) :H_{\DR}(\rho)\iso
\Hom_{A}(\Lambda, \CC_{\infty}).
\end{equation}
See \S\ref{sub:PeriodsQuasiP} for details about the de Rham group
and quasi-periodic functions.

Fixing a maximal $k$-linearly independent set $\{ \omega_{1},\ldots,
\omega_{r}\}$ of $\Lambda$ and a basis
$\{\delta_{1},\ldots,\delta_{r}\}$ of $H_{\DR}(\rho)$ defined over
$\ok$, the $r\times r$ matrix
$\rP_{\rho}:=\left(F_{\delta_{j}}(\omega_{i}) \right)$ is a
\emph{period matrix} of $\rho$.  Although the matrix $\rP_\rho$ is
not uniquely defined, the field $\ok(\rP_\rho)$ is independent of
the choices made.  Much as in the case of elliptic curves,
$\delta_1, \dots, \delta_r$ can be chosen appropriately so that the
first column of $\rP_\rho$ consists of elements of $\Lambda$
(periods of the first kind) and the entries of the remaining columns
are called quasi-periods (periods of the second kind).

Each entry of $\rP_{\rho}$ is transcendental over $k$ by the work of
Yu \cite{Yu86, Yu90}. Let $s$ be the degree of $K_\rho$ over $k$.
Using Yu's sub-$t$-module theorem \cite[Thm.~0.1]{Yu97}, Brownawell
proved that all the $\ok$-linear relations among the entries of
$\rP_{\rho}$ are those induced from endomorphisms of $\rho$; in
particular, the dimension of the $\ok$-vector space spanned by the
entries of $\rP_{\rho}$ is $r^{2}/s$ (cf.
\cite[Prop.~2]{Brownawell}). The period conjecture of Brownawell-Yu
asserts that among the entries of $\rP_{\rho}$, these linear
relations account for all the $\ok$-algebraic relations.  The second
main theorem of the present paper is to prove this assertion.

\begin{theorem}\label{T:Thm2Introd}
Let $\rho$ be a Drinfeld $A$-module of rank $r$ defined over $\ok$,
and let $s$ be the degree of $K_\rho$ over $k$. Let $\rP_{\rho}$ be
the period matrix of $\rho$. Then we have
\[
\trdeg_{\ok} \ok(\rP_{\rho})=r^{2}/s.
\]
\end{theorem}

Special cases of this result in rank $2$ were previously established
by Thiery \cite{Thiery} (algebraic independence in the CM case),
David and Denis \cite{DavidDenis} (quadratic independence), and the
authors \cite{CP} (algebraic independence in the non-CM, odd
characteristic case).  Based on this theorem, we further prove the
following algebraic independence result on Drinfeld quasi-logarithms. Theorem~\ref{T:Thm3Introd} has
applications to algebraic independence results about periods of the
first, second, and third kind for rank $2$ Drinfeld modules
\cite{Chang}.

\begin{theorem}\label{T:Thm3Introd}
Let $\rho$ be a Drinfeld $A$-module of rank $r$ defined over $\ok$.
Let $\delta_{1},\ldots,\delta_{r}$ be a basis of $H_{\DR}(\rho)$
defined over $\ok$. Let $u_{1},\ldots,u_{n}\in \CC_{\infty}$ satisfy
$\exp_{\rho}(u_{i})\in \ok$ for $i=1,\ldots,n$.  If
$u_{1},\ldots,u_{n}$ are linearly independent over $K_\rho$, then
the $rn$ quasi-logarithms $\cup_{i=1}^{n} \cup_{j=1}^{r} \left\{
F_{\delta_{j}}(u_{i}) \right\}$ are algebraically independent over
$\ok$.
\end{theorem}

\subsection{Outline}
To prove these results we travel the route along the deep connection
between Drinfeld modules and Anderson's theory of $t$-motives
\cite{Anderson86}.  One of our primary tools is the main theorem of
\cite{Papanikolas}, which asserts that the dimension of the
Tannakian Galois group of a $t$-motive, defined as the Galois group
of a system of difference equations, is equal to the transcendence
degree of its associated period matrix (see \S\ref{sec:t-motives}
for relevant background).  This theorem itself is rooted in a linear
independence criterion developed by Anderson, Brownawell, and the
second author \cite{ABP}. Thus the overall strategy for proving each
of Theorems~\ref{T:Thm1Introd}, \ref{T:Thm2Introd},
and~\ref{T:Thm3Introd} is to create a suitable $t$-motive such that
the special values in question are related to its period matrix and
then calculate its associated Galois group.

One of our goals in writing this paper has been to establish these
algebraic independence results as explicitly as
possible, beyond what is supplied in the general theory of
\cite{Papanikolas}.  If one is purely interested in only the
transcendence degree of these values, then it suffices by
\cite{Papanikolas} to calculate the dimension of the Galois group of
the $t$-motive.   Then using upper bounds on the transcendence
degree coming from well-known $\ok$-linear relations among periods,
quasi-periods, and logarithms of Drinfeld modules, one can in
principle obtain the results of the present paper simply from the
dimension of the Galois group.  However, by keeping track of
detailed information about these $t$-motives and the solutions of
their associated difference equations, we are able to (1) compute
the associated Galois groups explicitly, and (2) demonstrate
explicitly the connection between these $t$-motives and the known
$K_\rho$-linear relations among periods and logarithms (see
\S\ref{sec:L.IndepofMi}--\ref{sec:DrinLogs}).  Thus, we recover
linear independence results of Brownawell~\cite{Brownawell} and
Yu~\cite{Yu97} in the process, but without the need to appeal to the
theory of $t$-modules as Yu's sub-$t$-module theorem requires.
Furthermore, the techniques of proof are more robust than the ones
presented in \cite{CP}, where the case of rank $2$ Drinfeld
$\FF_q[t]$-modules, without CM and in odd characteristic, was
considered.  We have endeavored to highlight the advances beyond
\cite{CP} in this paper.

Throughout this paper we first consider the case when $A$ is a
polynomial ring, and then extend the results to general $A$ via
complex multiplication. In \S\ref{sec:DrinModules}, we study the
$t$-motive $M_{\rho}$ associated to a given Drinfeld module $\rho$
and its Galois group $\Gamma_{M_{\rho}}$. The main result of
\S\ref{sec:DrinModules} is to use Anderson generating functions to
prove that the image of the Galois representation on the $t$-adic
Tate module of $\rho$ is naturally contained inside the $t$-adic
valued points of $\Gamma_{M_{\rho}}$. Therefore, using a fundamental
theorem of Pink~\cite{Pink97} on the openness of the image of the
Galois representation on the $t$-adic Tate module of $\rho$, we
obtain an explicit description of $\Gamma_{M_{\rho}}$ (Theorem
\ref{T:MainThm1}). This enables us to prove Theorem
\ref{T:Thm2Introd}.

Let $\cT$ denote the category of $t$-motives.  Given
$u_{1},\ldots,u_{n}$ as in Theorems~\ref{T:Thm1Introd}
and~\ref{T:Thm3Introd}, in \S\ref{sec:L.IndepofMi} we construct
$t$-motives $X_{1}, \ldots, X_{n}$ representing classes in
$\Ext^1_{\cT}(\one,M_{\rho})$ such that the union of the entries of
their period matrices contains $\cup_{i=1}^{n} \cup_{j=1}^{r} \{
F_{\delta_{j}}(u_{i})\}$.  If $u_1, \dots, u_n$ together with a
$K_\rho$-basis of $k \otimes_{A} \Lambda_\rho$ are linearly
independent over $K_\rho$, then we use techniques of Frobenius
difference equations  to prove that $X_1, \dots, X_n$ are
$\End_{\cT}(M_\rho)$-linearly independent in
$\Ext_{\cT}^1(\one,M_{\rho})$ (Theorem~\ref{T:IndMiExt}). As
observed by Hardouin \cite{Hardouin}, the
$\End_{\cT}(M_{\rho})$-linear independence of $X_{1},\ldots,X_{n}$
provides information for the dimension of the Galois group of
$\oplus_{i=1}^{n}X_{i}$. Based on Theorem~\ref{T:IndMiExt}, in
\S\ref{sub:CalGalGroup} we give a detailed proof of Theorem
\ref{T:Thm3Introd} in the polynomial ring case
(Corollary~\ref{C:CorThm2}).  Using Theorem~\ref{T:MainThm1} and
Corollary~\ref{C:CorThm2}, we then prove the general cases of
Theorems~\ref{T:Thm1Introd}, \ref{T:Thm2Introd},
and~\ref{T:Thm3Introd} in \S\ref{sub:generalA}.

Note that in an earlier version of this paper,
Theorem~\ref{T:Thm3Introd} was worked out by the authors under the
assumption that $K_{\rho}$ is separable over $k$.  This
separability hypothesis boiled down to a question on algebraic groups
(cf. Lemma~\ref{L:SmoothG}), and the purpose of Appendix~\ref{appendix}
provided by B.~Conrad is to remove the hypothesis of separability.

\subsection*{Acknowledgements} We thank J.~Yu for many helpful discussions, suggestions and encouragement throughout this project.
We particularly thank NCTS for financial support so that we were
able to visit each other over the past several years.  We further thank B.~Conrad for providing us with the appendix, crucially helping us handle inseparability issues occurring in the case of quasi-logarithms. Finally, we thank the referees for several helpful suggestions.

\section{$t$-motives and difference Galois
groups}\label{sec:t-motives}

\subsection{Notation and preliminaries}
Until \S\ref{sec:DrinLogs}, where we treat the case of general $A$,
we will restrict our base ring $A$ to be the polynomial ring.  We
adopt the following notation.
\begin{longtable}{p{0.5truein}@{\hspace{5pt}$=$\hspace{5pt}}p{5truein}}
$\FF_q$ & the finite field with $q$ elements, for $q$ a power of a
prime number $p$. \\
$\theta$, $t$, $z$ & independent variables. \\
$A$ & $\FF_q[\theta]$, the polynomial ring in $\theta$ over $\FF_q$.
\\
$k$ & $\FF_q(\theta)$, the fraction field of $A$.\\
$k_\infty$ & $\laurent{\FF_q}{1/\theta}$, the completion of $k$ with
respect to the place at infinity.\\
$\overline{k_\infty}$ & a fixed algebraic closure of $k_\infty$.\\
$\ok$ & the algebraic closure of $k$ in $\overline{k_\infty}$.\\
$\CC_\infty$ & the completion of $\overline{k_\infty}$ with respect
to
the canonical extension of $\infty$.\\
$\eA$ & $\FF_q[t]$, the polynomial ring in $t$ over $\FF_q$.\\
$\ek$ & $\FF_q(t)$, the fraction field of $\eA$.\\
$\TT$ & $\{ f \in \power{\CC_\infty}{t} \mid \textnormal{$f$
converges
on $|t|_\infty \leq 1$} \}$, the Tate algebra over $\CC_\infty$.\\
$\LL$ & the fraction field of $\TT$.\\
${\GL_{r}}/F$ & for a field $F$, the $F$-group scheme of invertible
$r \times r$ matrices.
\end{longtable}

For $n \in \ZZ$, given a Laurent series $f = \sum_i a_i t^i \in
\laurent{\CC_\infty}{t}$, we define the $n$-fold twist of $f$ by
$f^{(n)} = \sum_i a_i^{q^n} t^i$.  For each $n$, the twisting
operation is an automorphism of $\laurent{\CC_\infty}{t}$ and
stabilizes several subrings, e.g., $\power{\ok}{t}$, $\ok[t]$,
$\TT$, and $\LL$.  For any matrix $B$ with entries in
$\laurent{\CC_\infty}{t}$, we define $B^{(n)}$ by the rule
${B^{(n)}}_{ij} = B_{ij}^{(n)}$.  Also we note (cf.\
\cite[Lem.~3.3.2]{Papanikolas})
\[
  \FF_q[t] = \{ f \in \TT \mid f^{(-1)} = f \}, \quad \FF_q(t) = \{
  f \in \LL \mid f^{(-1)} = f \}.
\]

Given a ring $R \subseteq \laurent{\CC_\infty}{t}$ that is invariant
under $n$-fold twisting for all $n$, we define twisted polynomial
rings $R[\sigma]$, $R[\sigma^{-1}]$, and $R[\sigma,\sigma^{-1}]$,
subject to the relations
\[
  \sigma^{i} f = f^{(-i)}\sigma^{i}, \quad f \in R, i \in \ZZ.
\]
As a matter of notation, we will often write $\tau$ for
$\sigma^{-1}$, and so $R[\tau] = R[\sigma^{-1}]$.  If $R$ itself is
a polynomial ring in $t$, say $R = S[t]$, then we will write
$R[\sigma] = S[t,\sigma]$ instead of $S[t][\sigma]$.

\subsection{Pre-$t$-motives and $t$-motives}
We briefly review the definitions and results we will need about
pre-$t$-motives and $t$-motives.  The reader is directed to \cite[\S
3]{Papanikolas} for more details.  A \emph{pre-$t$-motive} $M$ is a
left $\ok(t)[\sigma,\sigma^{-1}]$-module that is finite dimensional
over $\ok(t)$.  If $\bm \in \Mat_{r\times 1}(M)$ is a $\ok(t)$-basis
of $M$, then there is a matrix $\Phi \in \GL_r(\ok(t))$ so that
$\sigma \bm = \Phi \bm$.  We say that $M$ is \emph{rigid
analytically trivial} if there exists $\Psi \in \GL_r(\LL)$ so that
\[
  \Psi^{(-1)} = \Phi\Psi.
\]
If we let $\sigma$ act diagonally on $\LL \otimes_{\ok(t)} M$ and
let $M^{\rB}$ be the $\ek$-subspace fixed by $\sigma$, then $M$ is
rigid analytically trivial if and only if $\dim_{\ek} M^{\rB} = r$.
In this case, the entries of $\Psi^{-1}\bm$ form a $\ek$-basis of
$M^{\rB}$ \cite[Prop.~3.4.7]{Papanikolas}.  By
\cite[Thm.~3.3.15]{Papanikolas} the category of rigid analytically
trivial pre-$t$-motives, denoted by $\cR$, is a neutral Tannakian
category over $\ek$ with fiber functor $M \mapsto M^{\rB}$ (see
\cite[\S II]{DMOS} for more details on Tannakian categories).  Its
trivial object is denoted by~$\one$.  For a pre-$t$-motive $M\in
\cR$, we let $\cR_M$ denote the strictly full Tannakian subcategory
generated by $M$.  In this way, $\cR_M$ is equivalent to the
category of representations over $\ek$ of an affine algebraic group
scheme $\Gamma_M$ over $\ek$, i.e., $\cR_M \approx
\Rep{\Gamma_M}{\ek}$.  The group $\Gamma_M$ is called the
\emph{Galois group of $M$}.

Following \cite{Anderson86,ABP,Papanikolas}, an \emph{Anderson
$t$-motive} $\eM$ is a left $\ok[t,\sigma]$-module which is free and
finitely generated as both a left $\ok[t]$-module and a left
$\ok[\sigma]$-module and which satisfies, for $n$ sufficiently
large, $(t-\theta)^n \eM \subseteq \sigma \eM$.  The functor
\[
  \eM \mapsto \ok(t) \otimes_{\ok[t]} \eM,
\]
from Anderson $t$-motives to pre-$t$-motives is fully faithful up to
isogeny (see \cite[Thm.~3.4.9]{Papanikolas}), and if $\ok(t)
\otimes_{\ok[t]} \eM$ is rigid analytically trivial, then for a
$\ok[t]$-basis $\bm$ of $\eM$ it is possible to pick a rigid
analytic trivialization $\Psi$ that lies in $\GL_r(\TT)$ (see
\cite[Prop.~3.4.7]{Papanikolas}).  By definition, the category $\cT$
of $t$-motives is the strictly full Tannakian subcategory generated
by Anderson $t$-motives in the category of rigid analytically
trivial pre-$t$-motives.

\subsection{Galois groups and difference equations}
Following \cite[\S 4--5]{Papanikolas} the Galois groups of
$t$-motives can be constructed explicitly using the theory of
Frobenius semi-linear difference equations.  Specifically, suppose
we are given a triple of fields $F \subseteq K \subseteq L$ together
with an automorphism $\sigma_*: L \to L$ that satisfy (1) $\sigma_*$
restricts to an automorphism of $K$ and is the identity on $F$; (2)
$F = K^{\sigma_*} = L^{\sigma_*}$; and (3) $L$ is a separable
extension of $K$.  Then given $\Phi \in \GL_r(K)$, $\Psi \in
\GL_r(L)$ satisfying $\sigma_* \Psi = \Phi\Psi$, we let $Z_\Psi$ be
the smallest closed subscheme of $\GL_r/K$ that contains $\Psi$ as
an $L$-rational point.  That is, if $K[X,1/\det X]$, $X = (X_{ij})$,
is the coordinate ring of $\GL_r/K$, then the defining ideal of
$Z_\Psi$ is the kernel of the $K$-algebra homomorphism
\[
  X_{ij} \mapsto \Psi_{ij} : K[X,1/\det X] \to L.
\]
Then if we let $\Gamma_\Psi$ be the smallest closed subscheme of
$\GL_r/F$ so that $\Gamma_\Psi(\oL) \supseteq
\Psi^{-1}Z_{\Psi}(\oL)$, the following properties hold.

\begin{theorem}[{Papanikolas \cite[Thm.~4.2.11,
Thm.~4.3.1]{Papanikolas}}] \label{T:DiffGalGrp} The scheme
$\Gamma_\Psi$ is a closed $F$-subgroup scheme of $\GL_r/F$, and the
closed $K$-subscheme $Z_\Psi$ of $\GL_r/K$ is stable under
right-multiplication by $K\times_F\Gamma_\Psi$ and is a $(K\times_F
\Gamma_\Psi)$-torsor, and in particular
\[
  \Gamma_\Psi(\oL) = \Psi^{-1}Z_{\Psi}(\oL) .
\]
In addition, if $K$ is algebraically closed in $K(\Psi) \subseteq
L$, then
\begin{enumerate}
\item[(a)] The $K$-scheme $Z_\Psi$ is smooth and
geometrically connected.
\item[(b)] The $F$-scheme $\Gamma_\Psi$ is smooth and
geometrically connected.
\item[(c)] The dimension of $\Gamma_{\Psi}$ over $F$ is equal to
the transcendence degree of $K(\Psi)$ over $K$.
\end{enumerate}
\end{theorem}

Finally we return to $t$-motives.  Suppose that $M \in \cR$, that
$\Phi \in \GL_r(\ok(t))$ represents multiplication by $\sigma$ on
$M$, and that $\Psi \in \GL_r(\LL)$ satisfies $\Psi^{(-1)} =
\Phi\Psi$.  Then using the triple of fields $(F,K,L) =
(\FF_q(t),\ok(t),\LL)$, we can construct the group $\Gamma_\Psi$,
which is a subgroup of $\GL_r/\FF_q(t)$.  Given any object $N$ of
$\cT_M$, one can construct a canonical representation of
$\Gamma_\Psi$ on $N^{\rB}$ (see \cite[\S 4.5]{Papanikolas}).  This
permits the identification of $\Gamma_M$ and $\Gamma_\Psi$.

\begin{theorem}[{Papanikolas
\cite[Thm.~4.5.10]{Papanikolas}}]\label{T:TannFiber} Given a
pre-$t$-motive $M$ together with $\Phi$, $\Psi$ as in the paragraph
above, the evident functor
\[
  N \mapsto N^{\rB} : \cR_M \to \Rep{\Gamma_{\Psi}}{\FF_q(t)}
\]
is an equivalence of Tannakian categories.  Moreover, $\Gamma_M
\cong \Gamma_\Psi$ over $\FF_q(t)$.
\end{theorem}

Furthermore, if $M = \ok(t) \otimes_{\ok[t]} \eM$ for an Anderson
$t$-motive $\eM$, then by \cite[Prop.~3.3.9]{Papanikolas} we can
choose $\Psi$ to be in $\GL_r(\TT)$ so that the entries of $\Psi$
converge on all of $\CC_\infty$ \cite[Prop.~3.1.3]{ABP}.  The main
theorem of \cite{Papanikolas} is then the following.

\begin{theorem}[{Papanikolas \cite[Thm.~1.1.7]{Papanikolas}}]
\label{T:TrDegGalGrp} Let $M$ be a $t$-motive, and let $\Gamma_M$ be
its Galois group. Suppose that $\Phi \in \GL_r(\ok(t)) \cap
\Mat_r(\ok[t])$ represents multiplication by $\sigma$ on $M$ and
that $\det\Phi = c(t-\theta)^s$, $c \in \ok^\times$.  If $\Psi \in
\GL_r(\TT)$ is a rigid analytic trivialization for $\Phi$, then
$\trdeg_{\ok} \ok(\Psi(\theta)) = \dim \Gamma_M$.
\end{theorem}

\section{Algebraic independence of periods and
quasi-periods}\label{sec:DrinModules} In order to maintain
consistency between the notation of Drinfeld modules and
$t$-motives, we switch slightly from \S\ref{Sec:Introduction} and
discuss ``Drinfeld $\eA$-modules'' instead of ``Drinfeld
$A$-modules.''

\subsection{Periods and quasi-periods}\label{sub:PeriodsQuasiP}
We review briefly information about Drinfeld modules.  For complete
treatments the reader is directed to \cite[Ch.~3--4]{Goss} and
\cite[Ch.~2]{Thakur}.  A \emph{Drinfeld $\eA$-module $\rho$} is
defined to be an $\FF_q$-algebra homomorphism $\rho: \eA \to
\CC_\infty[\tau]$ defined so that ${\textnormal{Im}}(\rho)\nsubseteq
\CC_{\infty}$, and if $\rho_a = a_0 + a_1\tau + \dots + a_s \tau^s$,
for $a \in \eA$, then $a_0 = a(\theta)$.  Thus,
\begin{equation} \label{E:rhot}
  \rho_t = \theta + \kappa_1\tau + \dots + \kappa_r \tau^r,
\end{equation}
and we say that $r = \deg_\tau \rho_t$ is the \emph{rank} of $\rho$.
If $\rho(\eA) \subseteq K[\tau]$ for a field $K \subseteq
\CC_\infty$, then we say that $\rho$ is \emph{defined over $K$}.

The Drinfeld module $\rho$ provides $\CC_\infty$ with the structure
of an $\eA$-module via
\[
  a \cdot x = \rho_a(x), \quad \forall a\in \eA, x \in
  \CC_\infty,
\]
and we let $(\CC_\infty,\rho)$ denote $\CC_\infty$ together with
this $\eA$-module structure.  A \emph{morphism} of Drinfeld modules
$\rho \to \rho'$ is a twisted polynomial $b \in \CC_\infty[\tau]$
such that $b \rho_a = \rho'_a b$, for all $a \in \eA$, and in this
way $b$ induces an $\eA$-module homomorphism $b:(\CC_\infty,\rho)
\to (\CC_\infty,\rho')$. We call $b$ an isomorphism of $\rho$ if
$b\in \CC_{\infty}^{\times}$, and we say $b$ is defined over $\ok$
if $b \in \ok[\tau]$. The $\eA$-algebra of all endomorphisms of a
Drinfeld module $\rho$ is denoted $\End(\rho)$.

There is a unique $\FF_q$-linear power series with coefficients in
$\CC_\infty$,
\[
  \exp_\rho(z) = z + \sum_{i=1}^\infty \alpha_i z^{q^i},
\]
called the \emph{exponential function of $\rho$}, which is entire,
is surjective on $\CC_\infty$, and satisfies
$\exp_{\rho}(a(\theta) z) = \rho_a(\exp_{\rho}(z))$, for $a \in
\eA$, $z \in
  \CC_\infty$.  The kernel $\Lambda_\rho \subseteq \CC_\infty$ of
$\exp_\rho$ is a discrete and finitely generated $A$-submodule of
$\CC_\infty$ of rank~$r$; elements in $\Lambda_{\rho}$ are called
{\em{periods}} of $\rho$. Thus we have an isomorphism of
$\eA$-modules $\CC_\infty/\Lambda_\rho \cong (\CC_\infty,\rho)$,
where on the left-hand side $t$ acts by multiplication by $\theta$.
Furthermore, the map
\[
  \iota = (c_0 + \dots + c_m\tau^m \mapsto  c_0) : \End(\rho)
\to \{ c \in
  \CC_\infty \mid c \Lambda_\rho \subseteq \Lambda_\rho \},
\]
is an isomorphism. Throughout this paper, we identify $\End(\rho)$
with the image of $\iota$.

In analogy with the de Rham cohomology for elliptic curves,
Anderson, Deligne, Gekeler, and Yu developed a de Rham theory for
Drinfeld modules, which characterizes isomorphism classes of
extensions of Drinfeld modules by $\Ga$ (see
\cite{BP02,Gekeler89,Thakur,Yu90}).  Continuing with our choice of
Drinfeld $\eA$-module $\rho$ of rank $r$, an $\FF_q$-linear map
$\delta : \eA \to \CC_\infty[\tau]\tau$ is called a
\emph{biderivation} if $\delta_{ab} = a(\theta)\delta_b +
\delta_a\rho_b$ for all $a$, $b \in \eA$.  The set of all
biderivations $D(\rho)$ is a $\CC_\infty$-vector space.   A
biderivation $\delta$ is said to be \emph{inner} if there exists $m
\in \CC_\infty[\tau]$ so that $\delta_a = a(\theta)m - m \rho_a$ for
all $a \in \eA$, in which case we denote this biderivation by
$\delta^{(m)}$. As in \cite{BP02,Gekeler89,Yu90}, we have
\[
  D_{\mathrm{si}}(\rho) = \{ \delta^{(m)} \mid m \in
  \CC_\infty[\tau]\tau\}\ \textnormal{(strictly inner),}
  \quad H_{\DR}(\rho) = D(\rho)/D_{\mathrm{si}}(\rho)
  \ \textnormal{(de Rham),}
\]
and $H_{\DR}(\rho)$ is called the \emph{de Rham group of $\rho$}.

Given $\delta \in D(\rho)$, there is a unique power series
$F_\delta(z) = \sum_{i=1}^\infty c_i z^{q^i} \in
\power{\CC_\infty}{z}$ so that
\begin{equation}\label{E:DiffFdelta}
  F_\delta(a(\theta) z) - a(\theta)F_\delta(z) = \delta_a
  (\exp_\rho(z)), \quad \forall a \in \eA,
\end{equation}
called the \emph{quasi-periodic function} associated to $\delta$. It
is an entire function on $\CC_\infty$ and satisfies $F_\delta(z+
\omega) = F_\delta(z) + F_\delta(\omega)$, for all $\omega \in
\Lambda_\rho$.  The values $F_\delta(\omega)$ for $\omega \in
\Lambda_\rho$ are called \emph{quasi-periods of $\rho$}.  Since the
map $F_\delta|_{\Lambda_\rho}$ is $A$-linear, there is a
well-defined $\CC_\infty$-linear map,
\begin{equation} \label{E:deRham}
  \delta \mapsto (\omega \mapsto F_{\delta}(\omega)) :
  H_{\DR}(\rho) \to \Hom_A(\Lambda_{\rho},\CC_\infty),
\end{equation}
and this map is an isomorphism (see \cite{Gekeler89}).

Note that as $\eA=\FF_{q}[t]$, every biderivation is uniquely
determined by the image of~$t$.  Hence the $\CC_{\infty}$-vector
space $H_{\DR}(\rho)$ has a conveniently chosen basis, represented
by biderivations $\delta_1, \delta_2, \dots, \delta_{r}$: $\delta_1$
is the inner biderivation $\delta^{(1)} : a \mapsto a(\theta) -
\rho_a$ that generates the space of all inner biderivations modulo
$D_{\mathrm{si}}(\rho)$, and the biderivations $\delta_2, \dots,
\delta_{r}$ are defined by $\delta_i:t \mapsto \tau^{i-1}$.  Now we
have $F_{\delta^{(1)}}(z) = z - \exp_\rho(z)$, and so
$F_{\delta^{(1)}}(\omega) = \omega$ for all $\omega \in
\Lambda_{\rho}$.  Thus if we put $F_{\tau^i}(z) :=
F_{\delta_{i+1}}(z)$, $i=1,\dots, r-1$, and let $\omega_1,\dots,
\omega_r$ be an $A$-basis of $\Lambda_\rho$, then we can set
\begin{equation} \label{E:Prho}
\rP_\rho := \left( F_{\delta_j}(\omega_i) \right) = \begin{pmatrix}
\omega_1 & F_\tau(\omega_1) & \cdots & F_{\tau^{r-1}}(\omega_1) \\
\omega_2 & F_\tau(\omega_2) & \cdots & F_{\tau^{r-1}}(\omega_2) \\
\vdots & \vdots & & \vdots \\
\omega_r & F_\tau(\omega_r) & \cdots & F_{\tau^{r-1}}(\omega_r)
\end{pmatrix},
\end{equation}
which we refer to as \emph{the period matrix of $\rho$}.  The first
column contains periods (of the first kind) of $\rho$, while the
remaining columns contain quasi-periods (periods of the second
kind). Basic properties of biderivations show that the field
$\ok(\rP_\rho)$ depends only on the isomorphism class of $\rho$ and not on the choice of
basis for $\Lambda_\rho$ or even the choice of basis for
$H_{\DR}(\rho)$ defined over $\ok$ (i.e., $\delta(\eA) \subseteq
\ok[\tau]\tau$).

\subsection{The $t$-adic Tate module and Anderson generating
functions} \label{S:tadic} For any $a \in \eA$ the torsion
$\eA$-module $\rho[a] := \{ x \in \CC_\infty \mid \rho_a(x) = 0 \}$
is isomorphic to $(\eA/(a))^{\oplus r}$.  Thus if we let $v$ be any
monic irreducible polynomial in $\eA$, we can define the \emph{Tate
module $T_v(\rho)$} to be
\[
  T_v(\rho) := \varprojlim \rho[v^m] \cong \eA_v^{\oplus r}.
\]

Now assume that $\rho$ is defined over $K \subseteq \ok$. Every
element of $\rho[v^m]$ is separable over $K$, and so the absolute
Galois group $\Gal(K^{\sep}/K)$ of the separable closure of $K$
inside $\ok$ acts on $T_v(\rho)$, thus defining a representation
\[
  \varphi_v : \Gal(K^{\sep}/K) \to \Aut(T_v(\rho)) \cong
  \GL_r(\eA_v).
\]

Because it is well-suited to our purposes we now specialize to the
case that $v=t$.  Fixing an $A$-basis $\omega_1, \dots, \omega_r$ of
$\Lambda_\rho$, we define
\[
  \xi_{i,m} := \exp_\rho \biggl( \frac{\omega_i}{\theta^{m+1}}
  \biggr) \in \rho[t^{m+1}]
\]
for $1 \leq i \leq r$ and $m \geq 0$.  In this way we define an
$\eA_t$-basis $x_1, \dots, x_r$ of $T_t(\rho)$ by taking $x_i :=
(\xi_{i,0}, \xi_{i,1}, \xi_{i,2}, \ldots)$.  Thus for $\epsilon \in
\Gal(K^{\sep}/K)$ we can define $g_{\epsilon} \in
\GL_r(\power{\FF_q}{t})$ so that
\[
  \varphi_t(\epsilon) \bx = g_{\epsilon}\bx,
\]
where $\bx = [ x_1, \dots, x_r]^{\tr}$.

For each $i$, $1 \leq i \leq r$, we define an \emph{Anderson
generating function},
\begin{equation} \label{E:AndGen}
  f_i(t) := \sum_{m=0}^\infty \xi_{i,m} t^m = \sum_{m=0}^\infty
  \exp_\rho \biggl( \frac{\omega_i}{\theta^{m+1}} \biggr) t^m \in
  \power{K^{\sep}}{t}.
\end{equation}
The group $\Gal(K^{\sep}/K)$ acts on $\power{K^{\sep}}{t}$ by acting
on each coefficient, and we extend this action entry-wise to
matrices with entries in $\power{K^{\sep}}{t}$. The following lemma
and corollary show that the Galois action on $f_i$ and its Frobenius
twists as elements of $\power{K^{\sep}}{t}$ is compatible with its
action on them as elements of $T_t(\rho)$.

\begin{lemma} \label{L:GalComp}
Let $\bff = [ f_1, \dots, f_r ]^{\tr}$.  For any $\epsilon \in
\Gal(K^{\sep}/K)$, we have $\epsilon(\bff) = g_\epsilon \bff$, where
$\epsilon(\bff) = [\epsilon(f_1), \dots, \epsilon(f_r)]^{\tr}$.
\end{lemma}

\begin{proof}
Given $\ba = \sum_{\ell=0}^\infty a_\ell t^\ell \in \power{\FF_q}{t}
= \eA_t$, it is easy to see that for each $i$, $\ba \cdot x_i = (\ba
\cdot \xi_{i,0}, \ba \cdot \xi_{i,1}, \ba \cdot \xi_{i,2}, \ldots)$,
where for each $m \geq 0$,
\begin{equation} \label{E:alphaxi}
  \ba \cdot \xi_{i,m} = a_m \xi_{i,0} + a_{m-1} \xi_{i,1} + \dots
  + a_0 \xi_{i,m} \in \rho[t^{m+1}].
\end{equation}
Now fix any $1 \leq s \leq r$, and let $[h_{s,1}, \dots, h_{s,r}]
\in \Mat_{1 \times r}(\power{\FF_q}{t})$ be the $s$-th row of
$g_\epsilon$. Then
\[
\epsilon(x_s) = \varphi_t(\epsilon)(x_s) = \sum_{i=1}^r h_{s,i}
\cdot x_i.
\]
Thus if we write $h_{s,i} = \sum_{\ell=0}^\infty \gamma_{i,\ell}
t^\ell$ as an element of $\power{\FF_q}{t}$, we see from
\eqref{E:alphaxi} that the $(m+1)$-th entry of $\epsilon(x_s)$ is
\[
\epsilon(\xi_{s,m}) = \sum_{i=1}^r (\gamma_{i,m}\xi_{i,0} +
\gamma_{i,m-1}\xi_{i,1} + \dots + \gamma_{i,0}\xi_{i,m}).
\]
It follows that
\[
  \epsilon(f_s) = \sum_{m=0}^\infty \biggl( \sum_{i=1}^r
(\gamma_{i,m}\xi_{i,0} + \gamma_{i,m-1}\xi_{i,1} + \dots +
\gamma_{i,0}\xi_{i,m}) \biggr) t^m.
\]
By reversing the order of summation we see that $\epsilon(f_s)$ is
the same as $[h_{s,1},\dots,h_{s,r}] \bff$ via multiplication of
power series in $\power{K^{\sep}}{t}$.
\end{proof}

\begin{corollary} \label{C:GalUpsilon}
For $1 \leq i,j \leq r$ define $\Upsilon \in
\Mat_r(\power{K^{\sep}}{t})$ so that $\Upsilon_{ij} := f_i^{(j-1)}$.
Then for any $\epsilon \in \Gal(K^{\sep}/K)$,
\[
  \epsilon(\Upsilon^{(1)}) = g_\epsilon \Upsilon^{(1)}.
\]
\end{corollary}

\begin{proof}
The $j$-th column of $\Upsilon^{(1)}$ is simply $\bff^{(j)}$, where
$\bff$ is defined in Lemma~\ref{L:GalComp}.  Since for each $i,j$,
we have $\epsilon(f_i^{(j)}) = \epsilon(f_i)^{(j)}$, it follows from
Lemma~\ref{L:GalComp} that $\epsilon(\bff^{(j)}) =
\epsilon(\bff)^{(j)} = (g_\epsilon \bff)^{(j)}$. Since $g_\epsilon
\in \GL_r(\power{\FF_q}{t})$, we have $(g_\epsilon \bff)^{(j)}
= g_\epsilon \bff^{(j)}$.
\end{proof}

\subsection{Anderson $t$-motives associated to Drinfeld
modules}\label{Sec:AndMotDrin} Let us continue with our choice of
rank $r$ Drinfeld $\eA$-module $\rho$, defined as in \eqref{E:rhot}
over $\ok$.  As mentioned in \S\ref{sub:PeriodsQuasiP},
$\ok(P_{\rho})$ is unique up to isomorphisms of $\rho$ and so we
assume that $\kappa_r = 1$. The proofs in this section are
essentially identical to those in \cite[\S 2.4]{CP}, and we omit
them for brevity.

We associate an Anderson $t$-motive $\eM_\rho$ to $\rho$ in the
following way.  We let $\eM_\rho$ be isomorphic to the direct sum of
$r$ copies of $\ok[t]$, and we represent multiplication by $\sigma$
on $\eM_\rho$ with respect to the standard basis $m_1, \dots, m_r$
of $\eM_\rho$ by
\begin{equation} \label{E:Phirho}
  \Phi_\rho := \begin{pmatrix}
0 & 1 & \cdots & 0 \\
\vdots & \vdots & \ddots & \vdots \\
0 & 0 & \cdots & 1 \\
(t-\theta) & -\kappa_1^{(-1)} & \cdots & -\kappa_{r-1}^{(-r+1)}
\end{pmatrix}.
\end{equation}
Using a similar proof to \cite[Lem.~2.4.1]{CP} we find that
$\eM_\rho$ defines an Anderson $t$-motive.  As a
$\ok[\sigma]$-module, $\eM_\rho$ has rank $1$, and in fact $\eM_\rho
= \ok[\sigma]m_1$.  Finally we let $M_\rho := \ok(t)
\otimes_{\ok[t]} \eM_\rho$ be the pre-$t$-motive associated to
$\eM_\rho$.

Now a morphism $b : \rho \to \rho'$ of Drinfeld modules induces a
morphism $\beta : \eM_\rho \to \eM_{\rho'}$ of Anderson $t$-motives:
if $b = \sum c_i \tau^i$, then letting $b^* = \sum c_i^{(-i)}
\sigma^i \in \ok[\sigma]$, it follows, using methods similar to
\cite[Lem.~2.4.2]{CP}, that $\beta$ is the $\ok[\sigma]$-linear map
such that $\beta(m_1) = b^* m_1'$.  Moreover, we have the following
crucial result due to Anderson (see \cite[Prop.~2.4.3]{CP}).

\begin{proposition}\label{P:FunctorMovDrin}
The functor $\rho \mapsto \eM_\rho$ from Drinfeld $\eA$-modules over
$\ok$ to the category of Anderson $t$-motives is fully faithful.
Moreover, for any Drinfeld module $\rho$ over $\ok$,
\[
  \End(\rho) \cong \End_{\ok[t,\sigma]}(\eM_\rho), \quad K_\rho \cong
  \End_{\cT}(M_\rho).
\]
\end{proposition}

\begin{corollary}\label{C:SimpleMrho}
For a Drinfeld $\eA$-module $\rho$ over $\ok$, $M_{\rho}$ is a
simple left $\ok(t)[\sigma,\sigma^{-1}]$-module.
\end{corollary}

\begin{proof}
Every non-zero morphism of Drinfeld $\eA$-modules is surjective, and
therefore, every object in the category of Drinfeld $\eA$-modules is
simple. By Proposition~\ref{P:FunctorMovDrin}, $\eM_\rho$ is a
simple Anderson $t$-motive, and the result follows easily from
\cite[Prop.~4.4.10]{ABP}.
\end{proof}

\subsection{Drinfeld modules and rigid analytic
trivializations}\label{Sec:DrinRigid}
Given a Drinfeld $\eA$-module $\rho$ as in the previous section, we demonstrate here how its
associated Anderson $t$-motive $\eM_\rho$ is rigid analytically
trivial.  The arguments follow methods of Pellarin~\cite[\S
4.2]{Pellarin}.

For $u \in \CC_\infty$, we consider the Anderson generating function
\begin{equation}\label{E:AndGenFn}
  f_u(t) := \sum_{m=0}^\infty \exp_\rho \biggl(
  \frac{u}{\theta^{m+1}} \biggr) t^m = \sum_{i=0}^\infty
  \frac{\alpha_i u^{q^i}}{\theta^{q^i} - t} \in \TT,
\end{equation}
where $\exp_\rho(z) = z + \sum_{i=1}^\infty \alpha_i z^{q^i}$.  The
function $f_u(t)$ is meromorphic on $\CC_\infty$ with simple poles
at $t=\theta$, $\theta^q, \ldots$ with residues $-u$, $-\alpha_1
u^q, \ldots$.  Since $\rho_t(\exp_\rho(u/\theta^{m+1})) =
\exp_\rho(u/\theta^m)$,
\begin{equation} \label{E:fuFnEq}
  \kappa_1 f_u^{(1)}(t) + \dots + \kappa_{r-1} f_u^{(r-1)}(t) +
  f_u^{(r)} = (t-\theta)f_u(t) + \exp_\rho(u).
\end{equation}
It follows upon specializing at $t=\theta$ that
\begin{equation} \label{E:fuFnEqSpec}
  \kappa_1 f_u^{(1)}(\theta) + \dots + \kappa_{r-1}
  f_u^{(r-1)}(\theta) + f_u^{(r)}(\theta) = -u + \exp_\rho(u).
\end{equation}

\begin{lemma} \label{L:fuLinInd}
If $u_1, \dots, u_n \in \CC_\infty$ are linearly independent over
$k$, then the functions $f_{u_1}(t), \dots, f_{u_n}(t)$ are linearly
independent over $\ek$.
\end{lemma}

\begin{proof}
For $c_1(t), \dots, c_n(t) \in \ek$, the residue
$\Res_{t=\theta} \sum_{i=1}^n c_i(t)f_{u_i}(t) = -\sum_{i=1}^n c_i(\theta) u_i$.
\end{proof}

Once we fix an $A$-basis $\omega_1, \dots, \omega_r$ of the period
lattice $\Lambda_\rho$ of $\rho$, the Anderson generating functions
$f_1, \dots, f_r$ from \eqref{E:AndGen} are then $f_{\omega_1},
\dots, f_{\omega_r}$.  For $1 \leq i \leq r$ and $1 \leq j \leq
r-1$, it follows from \cite[p.~194]{Gekeler89}, \cite[\S
6.4]{Thakur} that
\begin{equation} \label{E:AndGenQuasi}
  F_{\tau^j}(\omega_i) = \sum_{m=0}^{\infty} \exp_\rho \biggl(
  \frac{\omega_i}{\theta^{m+1}} \biggr)^{q^j} \theta^m =
  f_i^{(j)}(\theta).
\end{equation}
Define the matrix
\[
  \Upsilon = \begin{pmatrix}
  f_1 & f_1^{(1)} & \cdots & f_1^{(r-1)} \\
  f_2 & f_2^{(1)} & \cdots & f_2^{(r-1)} \\
  \vdots & \vdots & & \vdots \\
  f_r & f_r^{(1)} & \cdots & f_r^{(r-1)}
\end{pmatrix}.
\]
By Lemma~\ref{L:fuLinInd}, $f_1, \dots, f_r$ are linearly
independent over $\ek$, and so if one argues as in
\cite[Lem.~1.3.3]{Goss} (the case of Moore determinants), it follows
that $\det\Upsilon \neq 0$.  Now letting
\[
  \Theta = \begin{pmatrix}
0 &   \cdots & 0  & t-\theta \\
1 &   \cdots & 0  & -\kappa_1 \\
\vdots & \ddots & \vdots & \vdots \\
0 & \cdots & 1 & -\kappa_{r-1}
\end{pmatrix},
\]
we see from \eqref{E:fuFnEq} that $\Upsilon^{(1)} = \Upsilon\Theta$.
To create a rigid analytic trivialization for $\Phi_\rho$, we let
\[
  V := \begin{pmatrix}
  \kappa_1 & \kappa_2^{(-1)} & \cdots & \kappa_{r-1}^{(-r+2)} & 1\\
  \kappa_2 & \kappa_3^{(-1)} & \cdots & 1 &\\
  \vdots & \vdots & & & \\
  \kappa_{r-1} & 1 & & & \\
  1 & & & &
  \end{pmatrix},
\]
and then set
\begin{equation} \label{E:PsiDef}
  \Psi_\rho := V^{-1} \bigl[ \Upsilon^{(1)} ]^{-1}.
\end{equation}
Since $V^{(-1)} \Phi_\rho = \Theta V$ and $\Upsilon^{(1)} =
\Upsilon\Theta$, it follows that $\Psi_\rho^{(-1)} = \Phi_\rho
\Psi_\rho$.  Thus the pre-$t$-motive $M_\rho$ is rigid analytically
trivial and is in the category of $t$-motives
(cf.~\cite[Prop.~3.4.7(c)]{Papanikolas}).

\begin{proposition} \label{P:PsiPer}
Given a Drinfeld $\eA$-module $\rho$ defined over $\ok$, the matrix
$\Psi_\rho$ defined in \eqref{E:PsiDef} and the period matrix
$\rP_\rho$ defined in \eqref{E:Prho} satisfy the following
properties.
\begin{enumerate}
\item[(a)] The entries of $\Psi_\rho$ are regular at $t=\theta$.
\item[(b)] $\trdeg_{\ok}\, \ok(\Psi_\rho(\theta)) = \dim
\Gamma_{\Psi_\rho}$.
\item[(c)] $\ok(\Psi_\rho(\theta)) = \ok(\rP_\rho)$.
\end{enumerate}
\end{proposition}

\begin{proof}
By \cite[Prop.~3.3.9(c), \S 4.1.6]{Papanikolas}, there exists a
matrix $U \in \GL_r(\ek)$ such that $\tPsi := \Psi_\rho U$ is a
rigid analytic trivialization of $\Phi$ and $\tPsi \in \GL_r(\TT)$.
By \cite[Prop.~3.1.3]{ABP}, the entries of $\tPsi$ converge on all
of $\CC_\infty$. Part (a) then follows since the entries of $\tPsi$
and $U^{-1}$ are all regular at $t = \theta$.  Part (b) follows
directly from Theorem~\ref{T:TrDegGalGrp}, since
$\ok(\Psi_\rho(\theta)) = \ok(\tPsi(\theta))$.  To prove (c) we
first observe that $\ok(\Psi_\rho(\theta)) =
\ok(\Upsilon^{(1)}(\theta))$ from \eqref{E:PsiDef}.  For $1 \leq i
\leq r$ and $1 \leq j \leq r-1$, it follows from
\eqref{E:AndGenQuasi} that $\Upsilon^{(1)}(\theta)_{ij} =
F_{\tau^j}(\omega_i)$.  From \eqref{E:fuFnEqSpec} and
\eqref{E:AndGenQuasi} we see that
\begin{equation}\label{E:Upsilon1}
  \Upsilon^{(1)}(\theta)_{ir} = -\omega_i -\sum_{s=1}^{r-1} \kappa_s
  F_{\tau^s}(\omega_i),
\end{equation}
and thus $\ok(\Upsilon^{(1)}(\theta)) = \ok(\rP_\rho)$.
\end{proof}

\subsection{The Galois group $\Gamma_{\Psi_\rho}$}
The following theorem demonstrates an explicit connection between
the Galois group $\Gamma_{\Psi_\rho}$ of a Drinfeld module $\rho$
arising from difference equations and the $t$-adic representation
attached to the Galois action on the torsion points of $\rho$.

\begin{theorem} \label{T:tadictoDiff}
Let $\rho$ be a Drinfeld $\eA$-module that is defined over a
field $K \subseteq \ok$ such that $K$ is a finite extension of $k$;
$\End(\rho) \subseteq K[\tau]$; and $Z_{\Psi_\rho}$ is defined over
$K(t)$.  Then
\[
  \varphi_t(\Gal(K^{\sep}/K)) \subseteq
  \Gamma_{\Psi_\rho}(\laurent{\FF_q}{t}).
\]
\end{theorem}

\begin{proof}
Let $\epsilon \in \Gal(K^{\sep}/K)$, and let $g_\epsilon \in
\GL_r(\power{\FF_q}{t})$ be defined as in \S\ref{S:tadic}.  Then by
\eqref{E:PsiDef} and Corollary~\ref{C:GalUpsilon},
\begin{equation} \label{E:epsilonPsi}
  \epsilon(\Psi_\rho) = \epsilon \bigl(V^{-1} \bigl[\Upsilon^{(1)}
  \bigr]^{-1} \bigr) = V^{-1} \bigl[ g_{\epsilon} \Upsilon^{(1)}
  \bigr]^{-1} = \Psi_\rho g_{\epsilon}^{-1}.
\end{equation}
Now let $S \subseteq K(t)[X,1/\det X]$ denote a finite set of
generators of the defining ideal of $Z_{\Psi_\rho}$.  Thus for any
$h \in S$, we have $h(\Psi_\rho) = 0$.  Since $\epsilon$ fixes the
coefficients of $h$, we have
\[
  0 = \epsilon(h(\Psi_\rho)) = h(\Psi_\rho g_{\epsilon}^{-1}).
\]
Therefore, $\Psi_\rho g_{\epsilon}^{-1} \in
Z_{\Psi_\rho}(\laurent{\CC_\infty}{t})$.  By
Theorem~\ref{T:DiffGalGrp}, we see that $g_{\epsilon}^{-1} \in
\Gamma_{\Psi_\rho}(\laurent{\FF_q}{t})$.
\end{proof}

\begin{remark}
Using basic properties of Drinfeld modules \cite[\S 4.7]{Goss} and
the fact that $Z_{\Psi_\rho}$ is of finite type over $\ok(t)$, we
can select a field $K \subseteq \ok$ satisfying the properties in
the previous theorem for any Drinfeld $\eA$-module $\rho$ defined
over $\ok$ .
\end{remark}

Combining Theorem~\ref{T:tadictoDiff} with theorems of Pink
\cite{Pink97} on the image of Galois representations attached to
Drinfeld modules, we calculate $\Gamma_{\Psi_\rho}$ exactly.  It is
also possible to identify the Galois group $\Gamma_{\Psi_\rho}$ with
the Hodge-Pink group \cite{PinkHS} (e.g.\ see \cite{Juschka} for
additional applications to pure $t$-motives), and we thank U.~Hartl
for pointing out the proper references to us.  However, here we have
a direct proof suited to our context.

\begin{theorem}\label{T:MainThm1}
Let $\rho$ be a Drinfeld $\eA$-module of rank $r$ defined over
$\ok$.  Set $\eK_{\rho}:=\End_{\cT}(M_{\rho})$, which embeds
naturally into $\End(M_{\rho}^{\rB}) = \Mat_r(\ek)$ by Theorem
\ref{T:TannFiber}.  Define the centralizer
$\Cent_{{\GL_r}/\ek}(\eK_{\rho})$ to be the algebraic group over $\ek$ such that
for any $\ek$-algebra $\eR$,
\[
\Cent_{{\GL_r}/\ek}
  (\eK_{\rho})(\eR) := \left\{ \gamma\in \GL_r(\eR) \mid \gamma g=g \gamma \textnormal{\
  for all\ }g\in \eR\otimes_{\ek} \eK_\rho\subseteq \Mat_r(\eR)   \right\}.
\]
Then
\[
  \Gamma_{\Psi_\rho} = \Cent_{{\GL_r}/\ek}(\eK_{\rho}).
\]
In particular, $\trdeg_{\ok}\, \ok(\rP_\rho) = r^2/s$, where $s =
[\eK_{\rho}:\ek]$.
\end{theorem}

\begin{proof}
By general Tannakian theory \cite[\S II]{DMOS}, the tautological
representation $\Gamma_{M} \hookrightarrow \GL(M^{\rB})$ is
functorial in $M$, and so there is a natural embedding
\begin{equation} \label{E:GammaMInj}
  \Gamma_{\Psi_\rho} \hookrightarrow \Cent_{{\GL_r}/\ek}
  (\eK_{\rho}).
\end{equation}
By \cite[Thm.~0.2]{Pink97}, the Zariski closure of
$\varphi_t(\Gal(K^{\sep}/K))$ inside $\GL_r(\laurent{\FF_q}{t})$ is
open in $\Cent_{\GL_r(\laurent{\FF_q}{t})}(\eK_{\rho})$ with respect
to the $t$-adic topology.  Thus from Theorem~\ref{T:tadictoDiff} and
\eqref{E:GammaMInj},
\[
  \dim \Gamma_{\Psi_\rho} = \dim \Cent_{{\GL_r}/\ek}(\eK_{\rho}).
\]
Since the defining polynomials of $\Cent_{{\GL_r}/\ek}(\eK_\rho)$
are degree one polynomials, it is connected, and hence
$\Gamma_{\Psi_\rho} = \Cent_{{\GL_r}/\ek}(\eK_{\rho})$.  By
Proposition~\ref{P:PsiPer} it now suffices to show that $\dim
\Gamma_{\Psi_\rho} = r^2/s$.  Since $\Gamma_{\Psi_\rho}$ is smooth
by Theorem~\ref{T:DiffGalGrp}(b), we have $\dim
\Gamma_{\Psi_\rho} = \dim \Lie \Gamma_{\Psi_\rho}$.  Now $\Lie
\Gamma_{\Psi_\rho} = \Cent_{\Mat_r(\ek)}(\eK_{\rho})$, and the
result follows from \cite[Thm.~3.15(3)]{FD}.
\end{proof}

\begin{corollary}\label{C:ResScalars}
For any $\ek$-algebra $\eR$, we have $\Gamma_{\Psi_{\rho}}(\eR)\cong
\GL_{r/s}(\eR\otimes_{\ek}\eK_{\rho})  $ naturally in $\eR$.
\end{corollary}

\begin{proof}
Since we have the natural embedding $\eK_{\rho}\hookrightarrow
\End(M_{\rho}^{\rB})$, $M_{\rho}^{\rB}$ is a $\eK_{\rho}$-vector
space of dimension $r/s$. As the group of $\eR$-valued points of
$\Cent_{{\GL_r}/\ek}(\eK_{\rho})$ is identified with $\{\phi\in
\Aut_{\eR}(\eR\otimes_{\ek}M_{\rho}^{\rB}) \mid \phi\ \textnormal{is
$\eR\otimes_{\ek}\eK_{\rho}$-linear} \}$, which itself is identified
with $\GL_{r/s}(\eR\otimes_{\ek} \eK_{\rho})$, the result follows
from Theorem \ref{T:MainThm1}.
\end{proof}

\begin{corollary}\label{C:CompReducible}
Every object $W\in \Rep{\Gamma_{M_{\rho}}}{\ek}$ is completely
reducible.
\end{corollary}

\begin{proof}
Let $\ek[\Gamma_{\Psi_{\rho}}(\ek) ]\subseteq \Mat_{r}(\ek)$ be the
$\ek$-algebra generated by all elements of $\Gamma_{\Psi}(\ek)$.
Since by Theorem \ref{T:MainThm1},
$\Gamma_{\Psi_{\rho}}(\ek)=\Cent_{\GL_{r}(\ek)}(\eK_{\rho})$, we
have $\ek[\Gamma_{\Psi_{\rho}}(\ek)]=
\Cent_{\Mat_{r}(\ek)}(\eK_{\rho}) $, which is a simple ring
(cf.~\cite[Thm. 3.15]{FD}). Regarding $W$ as a module over
$\Cent_{\Mat_{r}(\ek)}(\eK_{\rho})$, we have a decomposition $W\cong
\oplus W_{i}$ such that each $W_{i}$ is a simple
$\Cent_{\Mat_{r}(\ek)}(\eK_{\rho})$-module.

Now it suffices to show that each $W_{i}$ is an object in
$\Rep{\Gamma_{M_{\rho}}}{\ek}$.  To prove it, for any $\ek$-algebra
$\eR$ we first observe that $\eR\otimes_{\ek} W_{i}$ is a module
over $\eR\otimes_{\ek} \Cent_{\Mat_{r}(\ek)}(\eK_{\rho})$, which is
naturally isomorphic to $\Cent_{\Mat_{r}(\eR)}(\eR\otimes_{\ek}
\eK_{\rho})$. By Theorem \ref{T:MainThm1},
$\Gamma_{\Psi_{\rho}}(\eR)$ is equal to
$\Cent_{\GL_{r}(\eR)}(\eR\otimes_{\ek}  \eK_{\rho})$, which is the
group of units in $\Cent_{\Mat_{r}(\eR)}(\eR\otimes_{\ek}
\eK_{\rho})$. It follows that $\eR\otimes_{\ek} W_{i}$ is a
$\Gamma_{\Psi_{\rho}}(\eR)$-module, and hence $W_{i}\in
\Rep{\Gamma_{M_{\rho}}}{\ek}$.
\end{proof}

\section{Drinfeld logarithms and $\Ext_{\cT}^{1}(\one,M_{\rho})$}
\label{sec:L.IndepofMi} In this section, we fix a Drinfeld
$\eA$-module $\rho$ of rank $r$ defined over $\ok$ as in
\eqref{E:rhot} with $\kappa_{r}=1$, and we fix an $A$-basis
$\{\omega_{1},\ldots,\omega_{r}\}$ of $\Lambda_{\rho}$. We let
$M_{\rho}$ be the $t$-motive associated to $\rho$ together with a
fixed $\ok(t)$-basis $\bm \in \Mat_{r\times 1}(M_\rho)$,
$\Phi_{\rho}$ as in \eqref{E:Phirho}, and $\Psi_{\rho}$ as in
\eqref{E:PsiDef}. Finally, we let $K_\rho := k \otimes_A \End(\rho)$
and $\eK_\rho := \End_{\cT}(M_\rho)$.

\subsection{Endomorphisms of $t$-motives}
\begin{proposition}\label{P:F11}
Given $e\in \End_{\cT}(M_{\rho})$, let $E=(E_{ij})\in
\Mat_{r}(\ok(t))$ satisfy $e(\bm)=E \bm$.
\begin{enumerate}
\item [(a)] Each entry $E_{ij}$ is regular at $t=\theta$,
$\theta^{q}$, $\theta^{q^{2}}, \ldots$.
\item [(b)] $E_{21}(\theta)=\cdots=E_{r1}(\theta)=0$.
\item [(c)] $E_{11}(\theta)$ lies in $K_\rho$.
\end{enumerate}
\end{proposition}

\begin{proof}
In \cite[p.~146]{Papanikolas}, it has been shown that the
denominator of each $E_{ij}$ is in $\eA$, proving (a).  To prove
(b), define $\eta:=\Psi_{\rho}^{-1}E\Psi_{\rho}$. Using the equation
$\Phi_{\rho}E = E^{(-1)} \Phi_{\rho}$, we have $\eta^{(-1)}=\eta$
and hence $\eta\in {\Mat}_{r}(\ek)$. Using \eqref{E:PsiDef} and
\eqref{E:Upsilon1}, specializing at $t=\theta$ on both sides of
$\eta \Upsilon^{(1)}V=\Upsilon^{(1)}VE$ gives rise to the following
relation
\begin{equation} \label{E:etaomega}
\begin{aligned}
-\sum_{i=1}^{r} \eta_{1i}(\theta) \omega_{i} = &{}
-E_{11}(\theta)\omega_{1}+\bigl(E_{21}(\theta)\kappa_{2}^{(-1)}
+\dots + E_{r1}(\theta)
\bigr)F_{\tau}(\omega_{1})  \\
& {}+\bigl(
E_{21}(\theta)\kappa_{3}^{(-1)}
+\dots + E_{(r-1)1}(\theta)
\bigr)F_{\tau^{2}}(\omega_{1}) \\
& \qquad\qquad \vdots \\
&{} + \bigl( E_{21}(\theta)\kappa_{r-1}^{(-1)}+E_{31}(\theta) \bigr)
F_{\tau^{r-2}}(\omega_{1}) \\
&{} + E_{21}(\theta)F_{\tau^{r-1}}(\omega_{1}).
\end{aligned}
\end{equation}
Let $s = [K_\rho:k]$.  By Theorem~\ref{T:MainThm1} (or by
\cite[Prop.~2]{Brownawell}), we see that $\{F_{\tau}(\omega_{1}),
\ldots, F_{\tau^{r-1}}(\omega_{1})\}$ joined with any maximal
$K_\rho$-linearly independent subset of $\Lambda_{\rho}$ are
themselves linearly independent over $\ok$.  It follows that for
each $1\leq i\leq r-1$ the coefficients of
$F_{\tau^{i}}(\omega_{1})$ in \eqref{E:etaomega} are zero. Thus we
have $E_{21}(\theta) = \dots = E_{r1}(\theta)=0$.  To prove (c),
specialize both sides of the first columns of $\eta \Upsilon^{(1)}V
= \Upsilon^{(1)}VE$ at $t=\theta$. Then by~(b),
\[
\eta(\theta)\begin{pmatrix}
-\omega_{1} \\
\vdots \\
-\omega_{r} \\
\end{pmatrix}
= E_{11}(\theta)
\begin{pmatrix}
-\omega_{1} \\
 \vdots \\
-\omega_{r} \\
\end{pmatrix}.
\]
Since $\eta(\theta)\in \Mat_{r}(k)$, the definition of $\End(\rho)$
implies that $E_{11}(\theta)$ falls in $K_\rho$.
\end{proof}

\subsection{The $\eK_{\rho}$-independence of $X_{i}$ in
$\Ext_{\cT}^{1}(\one, M_{\rho})$} \label{Sub:ObjExt} Given any $u\in
\CC_{\infty}$ with $\exp_{\rho}(u) = \alpha \in \ok$, we set $f_{u}$
to be the Anderson generating function of $u$ as in
\eqref{E:AndGenFn}. Let $\bh_{\alpha}$ be the column vector
$(\alpha,0,\ldots,0)^{\tr}\in \Mat_{r\times 1}(\ok)$. We define the
pre-$t$-motive $X_{\alpha}$ of dimension $r+1$ over $\ok(t)$ on
which multiplication by $\sigma$ is given by $\Phi_{\alpha}:=
\left(\begin{smallmatrix}
\Phi_{\rho} & \mathbf{0} \\
\bh_{\alpha}^{\tr} & 1
\end{smallmatrix} \right)$. We further define
\[
\bg_{\alpha}:=\begin{pmatrix}
-(t-\theta)f_{u}(t)-\alpha \\
-\bigl( \kappa_{2}^{(-1)}f_{u}^{(1)}(t) +
\dots+\kappa_{r-1}^{(-1)}f_{u}^{(r-2)}(t)+f_{u}^{(r-1)}(t)\bigr) \\
-\bigl(\kappa_{3}^{(-2)}f_{u}^{(1)}(t) + \dots +
\kappa_{r-1}^{(-2)}f_{u}^{(r-3)}(t)+f_{u}^{(r-2)}(t) \bigr) \\
\vdots \\
-f_{u}^{(1)}(t)
\end{pmatrix}
\in \Mat_{r\times 1}(\TT),
\]
and then we have the difference equation
$\Phi_{\rho}^{\tr}\bg_{\alpha}^{(-1)}=\bg_{\alpha}+\bh_{\alpha}$.
Putting
\begin{equation}\label{E:PsiAlpha}
\Psi_{\alpha}:=  \begin{pmatrix}
      \Psi_{\rho} & \mathbf{0} \\
      \bg_{\alpha}^{\tr}\Psi_{\rho} & 1 \\
    \end{pmatrix} \in \Mat_{r+1}(\TT ),
\end{equation}
we see that $\Psi_{\alpha}^{(-1)}=\Phi_{\alpha}\Psi_{\alpha}$, and
hence $X_{\alpha}$ is rigid analytically trivial. Based on
Proposition \ref{P:FunctorMovDrin}, the argument in
\cite[Prop.~6.1.3]{Papanikolas} shows that $X_{\alpha}$ is a
$t$-motive and so $X_{\alpha}$ represents a class in
$\Ext_{\cT}^{1}(\one,M_{\rho})$.

The group $\Ext^1_{\cT}(\one,M_{\rho})$ also has the structure of a
$\eK_\rho$-vector space by pushing out along $M_\rho$.  To see this
explicitly, suppose the operation of  $e \in \eK_\rho$ on $M_\rho$
is represented by $E \in \Mat_r(\ok(t))$ as in
Proposition~\ref{P:F11}. By choosing a $\ok(t)$-basis for an
extension $X$ so that multiplication by $\sigma$ on $X$ is
represented by $\left( \begin{smallmatrix} \Phi_\rho & 0 \\ \bv & 1
\end{smallmatrix}\right)$, we see that multiplication by $\sigma$ on
the push-out $e_*X$ is represented by $\left( \begin{smallmatrix}
\Phi_\rho & 0 \\ \bv E & 1 \end{smallmatrix} \right)$.  Likewise, as
is also standard, Baer sum in $\Ext^1_{\cT}(\one, M_{\rho})$ is
achieved by adding entries in the~$\bv$ row vector.

\begin{theorem}\label{T:IndMiExt} Suppose $u_{1}, \dots, u_{n}\in
\CC_{\infty}$ satisfy $\exp_{\rho}(u_{i}) = \alpha_{i}\in \ok$ for
$i=1, \dots, n$.  For each $i$, let $X_i := X_{\alpha_i}$ as above.
If $\dim_{K_\rho} \Span_{K_\rho} (\omega_1, \dots, \omega_r,
u_1,\dots, u_n) = r/s + n$, then the classes of $X_{1}, \dots,
X_{n}$ in $\Ext_{\cT}^1(\one,M_\rho)$ are linearly independent
over~$\eK_\rho$.
\end{theorem}

\begin{proof}
For each $i$ we define $\bh_{i}:=\bh_{\alpha_{i}}$,
$\bg_{i}:=\bg_{\alpha_{i}}$, $\Phi_{i}:=\Phi_{\alpha_{i}}$,
$\Psi_{i}:=\Psi_{\alpha_{i}}$ as in the preceding paragraphs.
Suppose that there exist $e_{1}, \dots, e_{n} \in \eK_{\rho}$, not
all zero, so that $N := {e_{1}}_{*}X_{1}+ \dots + {e_{n_{*}}}X_{n}$
is trivial in $\Ext_{\cT}^{1}(\one,M_{\rho})$.  Fix $E_{i}
\in\Mat_{r}(\ok(t))$ so that $e_{i}(\bm)=E_{i}\bm$ for each $i$.  By
choosing an appropriate $\ok(t)$-basis $\bn$ for~$N$, multiplication
by $\sigma$ on $N$ and a corresponding rigid analytic trivialization
are represented by matrices
\[
\Phi_{N}:= \begin{pmatrix}
      \Phi_{\rho} & \mathbf{0} \\
      \sum_{i=1}^{n}\bh_{i}^{\tr}E_{i} & 1
    \end{pmatrix}
\in \GL_{r+1}(\ok(t)), \quad \Psi_{N}:= \begin{pmatrix}
      \Psi_{\rho} & \mathbf{0} \\
    \sum_{i=1}^{n}\bg_{i}^{\tr}E_{i}\Psi_{\rho}   & 1
    \end{pmatrix}
\in \GL_{r+1}(\LL).
\]
Since $N$ is trivial in $\Ext_{\cT}^{1}(\one,M_{\rho})$, there
exists a matrix $\gamma= \left( \begin{smallmatrix} \Id_r & 0 \\
\gamma_1\ \cdots\ \gamma_r & 1
\end{smallmatrix} \right)
\in \GL_{r+1}(\ok(t))$, so that if we set $\bn':=\gamma \bn$, then
we have $ \sigma \bn'= (\Phi_{\rho}\oplus (1)) \bn'$, where
$\Phi_{\rho}\oplus (1) \in \GL_{r+1}(\ok(t))$ is the block diagonal
matrix with $\Phi_{\rho}$ and $1$ down the diagonal and zeros
elsewhere. Thus,
\begin{equation}\label{E:gammaPhiN}
\gamma^{(-1)}\Phi_{N}= \left(\Phi_{\rho}\oplus (1)\right) \gamma.
\end{equation}
Note that by \cite[p.~146]{Papanikolas} all denominators of $\gamma$
are in $\FF_{q}[t]$ and so in particular each $\gamma_{i}$ is
regular at $t=\theta$, $\theta^{q}$, $\theta^{q^{2}},\ldots$.

Using \eqref{E:gammaPhiN}, we have $(\gamma \Psi_{N})^{(-1)}=
(\Phi_{\rho}\oplus(1)) ( \gamma \Psi_{N})$.  By
\cite[\S4.1.6]{Papanikolas}, for some $\delta = \left(
\begin{smallmatrix} \Id_r & 0 \\ \delta_1\ \cdots\ \delta_r & 1
\end{smallmatrix} \right) \in \GL_{r+1}(\ek)$, we have $\gamma \Psi_{N}= (\Psi_{\rho}\oplus (1)) \delta$.
It follows that $(\gamma_{1},\dots,\gamma_{r}) +
\sum_{i=1}^{n}\bg_{i}^{\tr}E_{i} = (\delta_{1},\dots,\delta_{r})
\Psi_{\rho}^{-1}$.  Note that for each $i$, the first entry of
$\bg_{i}(\theta)$ is given by $u_{i}-\alpha_{i}$. Hence, using
Proposition \ref{P:F11}(b) and specializing both sides of this
equation at $t=\theta$, we obtain
\begin{equation}\label{E:Relgammaomega}
\gamma_{1}(\theta)+\sum_{i=1}^{n}(u_{i}-\alpha_{i})(E_i)_{11}(\theta)
= -\sum_{j=1}^{r}\delta_{j}(\theta)\omega_{j}.
\end{equation}
On the other hand, we find
$\gamma_{n}^{(-1)}(t-\theta)+\sum_{i=1}^{n}\alpha_{i} (E_i)_{11} =
\gamma_{1}$ from the $(r+1,1)$-entry of both sides of
\eqref{E:gammaPhiN}.  Specializing both sides of this equation at
$t=\theta$ implies $\gamma_{1}(\theta) =
\sum_{i=1}^n\alpha_{i}(E_i)_{11}(\theta)$.  Thus, from
\eqref{E:Relgammaomega} we obtain
\[
\sum_{i=1}^{n}(E_i)_{11}(\theta)u_{i}+\sum_{j=1}^{r}\delta_{j}(\theta)
\omega_{j} =0.
\]
The assumption that $e_{1}, \dots, e_{n}$ are not all zero implies
$E_{i}$ is nonzero for some $i$.  By Proposition
\ref{P:FunctorMovDrin} we have $\eK_{\rho}\cong K_{\rho}$ and
$E_{i}$ is invertible, and hence Proposition \ref{P:F11}(bc) implies
that $(E_i)_{11}(\theta) \in K_{\rho}^\times$, which gives the
desired contradiction.
\end{proof}

\section{Algebraic independence of Drinfeld logarithms}
\label{sec:DrinLogs}

\subsection{The case of $\eA = \FF_q[t]$} \label{sub:CalGalGroup}
We continue with the notation of \S\ref{sec:L.IndepofMi}.  We
suppose we have $u_1, \dots, u_n \in \CC_\infty$ with
$\exp_\rho(u_i) = \alpha_i \in \ok$ for each $i$, and then maintain
the choices from Theorem~\ref{T:IndMiExt} and its proof.  We note
that the rigid analytic trivializations $\Psi_i$ of each $X_i$ have
entire entries, which is assured by \cite[Prop.~3.1.3]{ABP}.  The
matrix $\Psi := \oplus_{i=1}^n \Psi_i$ provides a rigid analytic
trivialization for the $t$-motive $X := \oplus_{i=1}^n X_i$.  Our
goal in this section is to calculate the Galois group $\Gamma_X$.
{From} the construction of each $\Psi_i$, we see that
\[
\ok(\Psi(\theta))=\ok\left(
\cup_{i=1}^{r-1}\cup_{m=1}^{n}\cup_{j=1}^{r}\left\{\omega_{j},F_{\tau^{i}}(\omega_{j}),u_{m},F_{\tau^{i}}(u_{m})
\right\}  \right).
\]
By Theorem~\ref{T:TrDegGalGrp} we have $\dim
\Gamma_{\Psi}=\trdeg_{\ok}\ok(\Psi(\theta))$, and Theorem
\ref{T:MainThm1} then implies
\begin{equation}\label{E:dimtrdeg}
\dim \Gamma_{\Psi}=\trdeg_{\ok}\ok(\Psi(\theta))\leq
\frac{r^{2}}{s}+rn.
\end{equation}

Now let $N$ be the $t$-motive defined by $\Phi_N \in
\GL_{rn+1}(\ok(t))$ with rigid analytic trivialization $\Psi_N \in
\GL_{rn+1}(\TT)$:
\[
\Phi_{N}:= \begin{pmatrix}
\Phi_{\rho} &  & &  \\
& \ddots &  &  \\
&  & \Phi_{\rho} &  \\
\bh^{\tr}_{\alpha_{1}} & \cdots & \bh^{\tr}_{\alpha_{n}} & 1
\end{pmatrix},
\quad\Psi_{N}:=\begin{pmatrix}
\Psi_{\rho} &  &  &  \\
& \ddots &  &  \\
&  & \Psi_{\rho} &  \\
\bg^{\tr}_{\alpha_{1}}\Psi_{\rho} & \cdots
& \bg^{\tr}_{\alpha_{n}}\Psi_{\rho} & 1 \\
\end{pmatrix}.
\]
By the definition of $\Gamma_{\Psi_N}$, for any $\ek$-algebra $\eR$
each element $\nu \in \Gamma_{\Psi_N}(\eR)$ is of the form $\left(
\begin{smallmatrix} \oplus_{i=1}^n\gamma & 0 \\ * & 1
 \end{smallmatrix} \right)$, for some $\gamma\in
\Gamma_{\Psi_{\rho}}(\eR)$.  Note that $N$ is an extension of $\one$
by $M^{n}_{\rho}$, which is the pullback of $X \twoheadrightarrow
\one^{n}$ and the diagonal embedding $\one\hookrightarrow \one^{n}$.
Thus, as the two $t$-motives $X$ and $N$ generate the same Tannakian
sub-category of $\cT$, the Galois groups $\Gamma_{X}$ and
$\Gamma_{N}$ are isomorphic, in particular $\Gamma_{\Psi}\cong
\Gamma_{N}$. Note that by Theorem \ref{T:DiffGalGrp} any element
$\nu$ in $\Gamma_{N}$ is of the form
\[
\begin{pmatrix}
\gamma &  &  &  \\
& \ddots &  &  \\
&  & \gamma &  \\
\bv_{1} & \cdots & \bv_{n} & 1 \\
\end{pmatrix}
\] for some $\gamma\in \Gamma_{M_{\rho}}$ and some $\bv_1,\ldots,\bv_n\in
\Ga^{r}$. As $M_{\rho}^{n}$ is a sub-$t$-motive of $N$, we have the
short exact sequence of affine algebraic group schemes over $\ek$,
\begin{equation}\label{SESforGammaN}
         1 \to W \to \Gamma_{N} \stackrel{\pi}{\to} \Gamma_{M_{\rho}}
\to 1,
\end{equation}
where the surjective map $\pi: \Gamma_{N}\twoheadrightarrow
\Gamma_{M_{\rho}}$ is the projection map given by $\nu\mapsto \gamma
$ (cf.~\cite[p.22]{CP}). We notice that \eqref{SESforGammaN} gives
rise to an action of any $\gamma\in \Gamma_{M_{\rho}}$ on \[\bv=
\begin{pmatrix}
\Id_r &  &  &  \\
& \ddots &  &  \\
&  & \Id_r &  \\
\bv_{1} & \cdots & \bv_{n} & 1 \\
\end{pmatrix} \in W\]  given by
\[
\begin{pmatrix}
\Id_r &  &  &  \\
& \ddots &  &  \\
&  & \Id_r &  \\
\bv_{1}\gamma^{-1} & \cdots & \bv_{n}\gamma^{-1} & 1 \\
\end{pmatrix}  \]

In what follows, we will show that $W$ can be identified with a
$\Gamma_{M_{\rho}}$-submodule of $\left(M_{\rho}^{n}\right)^{\rB}\simeq
\left(M_{\rho}^{\rB} \right)^{n}$.  Let $\bn\in \Mat_{(rn+1)\times
  1}(N)$ be the $\ok(t)$-basis of $N$ such that $\sigma
\bn=\Phi_{N}\bn$. We write $\bn=(\bn_1,\ldots,\bn_n,x)^{\tr}$, where
$\bn_i\in \Mat_{r\times 1}(N)$. Notice that
$(\bn_1,\ldots,\bn_n)^{\tr}$ is a $\ok(t)$-basis of
$M_{\rho}^{n}$. Recall that the entries of $\Psi_{N}^{-1}\bn$ form a
$\ek$-basis of $N^{\rB}$, and
$\mathbf{u}:=(\Psi_{\rho}^{-1}\bn_{1},\ldots,\Psi_{\rho}^{-1}\bn_{n})^{\tr}$
is a $\ek$-basis of $M_{\rho}^{n}$. Given any $\ek$-algebra $\eR$, we
recall the action of $\Gamma_{M_{\rho}}(\eR)$ on
$\eR\otimes_{\ek}\left(M_{\rho}^{\rB}\right)^{n}$ as follows
(cf. \cite[\S 4.5]{Papanikolas}): for any $\gamma\in
\Gamma_{M_{\rho}}(\eR)$ and any $\bv_i \in \Mat_{1\times r}(\eR)$,
$1\leq i\leq n$, the action of~$\gamma$ on $\left( \bv_{1}, \ldots,
  \bv_{n}\right)\cdot \mathbf{u}\in \eR\otimes_{\ek}\left(
  M_{\rho}^{\rB} \right)^{n}$ is presented as
\[
\left( \bv_{1}, \ldots, \bv_{n}\right)\cdot \mathbf{u} \mapsto
\left( \bv_{1}\gamma^{-1}, \ldots, \bv_{n}\gamma^{-1} \right)\cdot\mathbf{u}.
\]
It follows that when a basis $\mathbf{u}$ of
$\left(M_{\rho}^{\rB}\right)^{n}$ is fixed as above, the action of
$\Gamma_{M_{\rho}}$ on $\left(M_{\rho}^{\rB}\right)^{n}$ is compatible
with the action of $\Gamma_{M_{\rho}}$ on $W$ described in the
previous paragraph.  Regarding $\left(M_{\rho}^{\rB}\right)^{n}$ as a
vector group over $\ek$, it follows that $\Gamma_{N}$ is a subgroup of
the group scheme $\Gamma_{M_{\rho}}\ltimes
\left(M_{\rho}^{\rB}\right)^{n}$ over $\ek$, which satisfies the short
exact sequence
\[
1\rightarrow
\left(M_{\rho}^{\rB}\right)^{n}\rightarrow \Gamma_{M_{\rho}}\ltimes
\left(M_{\rho}^{\rB}\right)^{n} \twoheadrightarrow
\Gamma_{M_{\rho}}\rightarrow 1 .
\]
We also see that $W$ is equal to
the scheme-theoretic intersection $\Gamma_{N}\cap
\left(M_{\rho}^{\rB}\right)^{n}$.

\begin{lemma}\label{L:SmoothG}
The $\ek$-group scheme $W$ is the $\ek$-vector subgroup of $\left(M_{\rho}^{\rB}\right)^{n}$
arising from a $\Gamma_{M_{\rho}}$-submodule.  In particular,
it is $\ek$-smooth.
\end{lemma}

\begin{proof}
  Recall that $\eK_{\rho}:=\End_{\cT}(M_{\rho})$ is naturally embedded
  into $\End(M_{\rho}^{B})$, so we can regard $M_{\rho}^{\rB}$ as
  vector space $\mathscr{V}'$ over $\eK_{\rho}$.
  Let  $G'$ denote the $\eK_{\rho}$-group $\GL(\mathscr{V}')$.  Corollary~\ref{C:ResScalars} says
  that $\Gamma_{M_{\rho}}$ naturally
  coincides with the Weil restriction of scalars $\mathrm{R}_{\eK_{\rho}/ \ek}(G')$
  over $\ek$. Let   $V'$ be the $\eK_{\rho}$-vector space
  $\mathscr{V}'^{\oplus n}$, and
 let $V$ be the vector group over $\ek$ associated to the underlying
 $\ek$-vector space. By Theorem~\ref{T:DiffGalGrp}(b), $\Gamma_{N}$ is
  smooth.  Hence, our situation is a special case of
  Example~\ref{exglv} of
  Appendix~\ref{appendix} (by taking $G=\Gamma_{M_{\rho}}$,
  $\Gamma=\Gamma_{N}$, $k'=\eK_{\rho}$, and $k=\ek$ there), so Proposition~\ref{glvprop} provides the
  desired result.
\end{proof}

\begin{remark}
In the case that $K_\rho$ is separable over $k$, one can prove
Lemma~\ref{L:SmoothG} by showing that the induced tangent map $d\pi$ is
surjective onto the Lie algebra of $\Gamma_{M_{\rho}}$ along the lines of
arguments of the proof of \cite[Prop.~4.1.2]{CP}.
\end{remark}

\begin{theorem}\label{T:MainThm2}
Given $u_{1},\ldots,u_{n}\in \CC_{\infty}$ with $\exp_{\rho}(u_{i})=
\alpha_{i}\in \ok$ for $i=1,\ldots,n$, we let $X_{1},\ldots,X_{n}\in
\Ext_{\cT}^{1}(\one,M_{\rho})$, $N$ and $\Psi$ be defined as above.
Suppose that  $ \omega_{1},\ldots,\omega_{r/s}$,
$u_{1},\ldots,u_{n}$ are linearly independent over $K_\rho$. Then
$\Gamma_{\Psi}$ is an extension of $\Gamma_{M_\rho}$ by a vector
group of dimension $rn$ defined over $\ek$, and so $\dim
\Gamma_{\Psi}=r(r/s+n)$. In particular, by \eqref{E:dimtrdeg}
\[
\cup_{i=1}^{r-1}\cup_{m=1}^{n}\cup_{j=1}^{r/s}
\left\{\omega_{j},F_{\tau^{i}}(\omega_{j}),u_{m},F_{\tau^{i}}(u_{m})
\right\}
\]
is an algebraically independent set over $\ok$.
\end{theorem}

\begin{proof}
By \eqref{E:dimtrdeg}, our task is to prove that $\dim
\Gamma_{\Psi}=r(r/s+n)$, which is equivalent to proving that $\dim
W=rn$ by Theorem \ref{T:MainThm1} and \eqref{SESforGammaN}.
As Lemma~\ref{L:SmoothG} implies that
 $W$ is a $\Gamma_{M_{\rho}}$-submodule of
$(M_{\rho}^{n})^{\rB}$,  by the equivalence of the
categories, $\cT_{M_{\rho}}\approx \Rep{\Gamma_{M_{\rho}}}{\ek}$,
there exists a sub-$t$-motive $U$ of $M_{\rho}^{n}$ so that
$W\cong U^{\rB}$. Therefore, to prove $\dim W=rn$ it suffices
to prove that $U^{\rB}=(M_{\rho}^{n})^{\rB}$.

We claim that $N/U$ is split as a direct sum of $M_{\rho}^{n}/U$ and
$\one$.  To prove this, we follow an argument of Hardouin (see
\cite[Lem.~2.3]{Hardouin}).  Since $W\cong U^{\rB}$, $\Gamma_{N}$
acts on $N^{\rB}/ U^{\rB}$ through the quotient $\Gamma_{N}/W\cong
\Gamma_{M_{\rho}}$ via \eqref{SESforGammaN}.  It follows that
$N^{\rB}/ U^{\rB}$ is an extension of $\ek$ by
$(M_{\rho}^{n})^{\rB}/ U^{\rB}$ in the category
$\Rep{\Gamma_{M_{\rho}}}{\ek}$. By the equivalence
$\cT_{M_{\rho}}\approx \Rep{\Gamma_{M_{\rho}}}{\ek}$  and Corollary
\ref{C:CompReducible}, we see that the extension $N/U$ is trivial in
$\Ext_{\cT}^{1}(\one,M_{\rho}^{n}/U)$.

Now suppose on the contrary that $U^{\rB}\subsetneq
(M_{\rho}^{n})^{\rB}$. By Corollary \ref{C:SimpleMrho} we have that
$M_{\rho}^{n}$ is completely reducible in $\cT_{M_\rho}$.  As $U$ is
a proper sub-$t$-motive of $M_{\rho}^{n}$, there exists a
non-trivial morphism $\phi\in \Hom_{\cT}(M_{\rho}^{n},M_{\rho})$ so
that $U\subseteq \Ker \phi$. Moreover, the morphism $\phi$ factors
through the map $M_{\rho}^{n}/U \rightarrow M_{\rho}^{n}/\Ker\phi$:
\[
\xymatrix{
M_{\rho}^{n} \ar[dr]^{\phi} \ar[d] & \\
M_{\rho}^{n}/ U \ar [r]& M_{\rho}^{n}/ \Ker \phi \cong M_{\rho}. }
\]
Since $\phi\in\Hom_{\cT}(M_{\rho}^{n},M_{\rho})$, we can write
$\phi(m_{1},\dots,m_{n})= \sum_{i=1}^{n} e_{i}(m_{i})$ for some
$e_{1}, \dots, e_{n}\in \eK_{\rho}$, not all zero. Then the push-out
$\phi_{*}N = {e_{1}}_{*}X_{1}+ \cdots + {e_{n}}_{*}X_{n}$ is a
quotient of $N/U$. By the claim above, it follows that $\phi_{*}N$
is trivial in $\Ext^{1}_{\cT}(\one,M_{\rho})$. But by Theorem
\ref{T:IndMiExt} this contradicts the $\eK_{\rho}$-linear
independence of $X_{1},\dots,X_{n}$ in
$\Ext^{1}_{\cT}(\one,M_{\rho})$.
\end{proof}

\begin{corollary}\label{C:CorThm2}
Let $\rho$ be a Drinfeld $\eA$-module of rank $r$ defined over
$\ok$. Let $u_{1},\ldots,u_{n}\in \CC_{\infty}$ satisfy
$\exp_{\rho}(u_{i})\in \ok$ for $i=1,\ldots,n$, and suppose that
they are linearly independent over $K_\rho$. Let
$\delta_{1},\ldots,\delta_{r}$ be a basis of $H_{\DR}(\rho)$ defined
over $\ok$.  Then the $rn$ quantities
\[
\cup_{j=1}^{r}\bigl\{F_{\delta_{j}}(u_{1}),\ldots,F_{\delta_{j}}(u_{n})
\bigr\}
\]
are algebraically independent over $\ok$.
\end{corollary}

\begin{proof}
Every Drinfeld $\eA$-module defined over $\ok$ is isomorphic over
$\ok$ to one for which the leading coefficient of $\rho_t$ is $1$.
It is a routine matter to check that the desired result is invariant
under isomorphism, so we will assume without loss of generality that
this leading coefficient is $1$.  Define the $K_{\rho}$-vector
space,
$V:=\Span_{K_\rho}(\omega_{1},\ldots,\omega_{r},u_{1},\ldots,u_{n})$,
and let $\{v_{1},\ldots,v_{\ell}\}$ be a $K_{\rho}$-basis of $V$.
Certainly $r/s \leq \ell \leq r/s+n$.  Since quasi-periodic
functions $F_\delta(z)$ are linear in $\delta$, using
\eqref{E:DiffFdelta} we have
\[
\ok\bigl(\cup_{i=1}^{r-1}\cup_{m=1}^{n}\cup_{j=1}^{r} \{\omega_{j},
F_{\tau^{i}}(\omega_{j}),u_{m},F_{\tau^{i}}(u_{m}) \}  \bigr)
=\ok\bigl( \cup_{j=1}^{r} \{
F_{\delta_{j}}(v_{1}),\ldots,F_{\delta_{j}}(v_{\ell})\} \bigr).
\]
By swapping out basis elements of $V$ as necessary (see the proof of
\cite[Thm.~4.3.3]{CP}), the result follows from
Theorem~\ref{T:MainThm2}.
\end{proof}

\subsection{The case of general $\eA$}\label{sub:generalA}
In this section, we prove Theorems~\ref{T:Thm1Introd},
\ref{T:Thm2Introd}, and~\ref{T:Thm3Introd}.  Essentially they follow
from Theorem~\ref{T:MainThm1} and Corollary~\ref{C:CorThm2} when we
consider $\rho$ to be a Drinfeld $\FF_q[t]$-module with complex
multiplication.

We resume the notation from the introduction.  However, in order to
separate the roles of ``$\eA$ as operators'' from ``$A$ as
scalars,'' we let $\eA$ be the ring of functions on $\eX$ that are
regular away from $\infty$ with fraction field $\ek$, and we let $A$
be a copy of $\eA$ with fraction field $k$ that serve as scalars.
The fields $k_\infty$ and $\CC_\infty$ are then extensions of $k$.
Thus we follow the ``two $t$'s'' convention of \cite[\S 5.4]{Goss}.

Let $\rho$ be a rank $r$ Drinfeld $\eA$-module defined over $\ok$.
Let $\eR$ be the endomorphism ring of $\rho$, considered to be an
extension of $\eA$, and let $\eK_\rho$ be the fraction field of
$\eR$.  We fix a non-constant element $t \in \eA$, and consider
$\rho$ to be a Drinfeld $\FF_q[t]$-module with complex
multiplication by $\eR$ and defined over $\ok$. Note that the
exponential function $\exp_\rho(z)$ and its period lattice remain unchanged when switching from $\eA$ to $\FF_q[t]$.

As for scalars, we let $\theta \in \ok$ be chosen so that the
extension $K_\rho/\FF_q(\theta)$ is canonically isomorphic to
$\eK_\rho/\FF_q(t)$, where as usual $K_\rho \subseteq \ok$ is the
fraction field of $\End(\rho)$.  In this way $\ok =
\overline{\FF_q(\theta)}$ and $\overline{k_\infty} =
\overline{\laurent{\FF_q}{1/\theta}}$.

\begin{proof}[Proof of Theorem~\ref{T:Thm1Introd}]  Since $u_1, \dots, u_n$ are assumed to
be linearly independent over $K_\rho$, Corollary~\ref{C:CorThm2}
dictates that, for any biderivation $\delta : \FF_q[t] \to
\ok[\tau]\tau$, $
  F_\delta(u_1), \dots, F_\delta(u_n)
$ are algebraically independent over $\ok$.  In particular, when we
take $\delta = \delta^{(1)}$, we know $F_{\delta^{(1)}}(z) = z -
\exp_{\rho}(z)$, and since $\exp_{\rho}(u_i) \in \ok$ for all $i$,
the theorem is proved.
\end{proof}

For the relevant background on the de Rham theory of Drinfeld
$\eA$-modules, we refer the reader to \cite{Gekeler89,Yu90}. In
order to distinguish the roles of $\rho$ over different base rings,
we denote by $D(\rho,\FF_q[t])$ (resp.\ $D(\rho,\eA)$) and
$H_{\DR}(\rho,\FF_q[t])$ (resp.\ $H_{\DR}(\rho,\eA)$), the spaces in
\S\ref{sub:PeriodsQuasiP} when we regard $\rho$ as a Drinfeld
$\FF_q[t]$-module (resp.\ Drinfeld $\eA$-module).  By the de Rham
isomorphism \eqref{E:deRhamIsomIntro} the restriction map $\delta
\mapsto \delta|_{\FF_q[t]} : D(\rho,\eA) \to D(\rho,\FF_q[t])$
induces an injection,
\[
  H_{\DR}(\rho,\eA) \hookrightarrow H_{\DR}(\rho,\FF_q[t]).
\]
Moreover, it is straightforward to check that $F_{\delta}(z) =
F_{\delta|_{\FF_q[t]}}(z)$ for all $\delta \in D(\rho,\eA)$.

\begin{proof}[Proofs of Theorems~\ref{T:Thm2Introd} and~\ref{T:Thm3Introd}]
Let $s: = [\eK_\rho:\ek]$,  $\ell := [\ek:\FF_q(t)]$.  Then $\rho$
is a Drinfeld $\FF_q[t]$-module of rank $r\ell$.  Selecting a basis
$\{ \delta_1, \dots, \delta_r\}$ for $H_{\DR}(\rho,\eA)$ defined
over $\ok$, we extend it to a basis $\{ \eta_1 :=
\delta_1|_{\FF_q[t]}, \dots, \eta_r := \delta_r|_{\FF_q[t]}, \eta_{r+1} \dots,
\eta_{r\ell} \}$ of $H_{\DR}(\rho,\FF_q[t])$ also defined over
$\ok$.  We first prove Theorem~\ref{T:Thm3Introd}, by applying Corollary~\ref{C:CorThm2} to $\rho$ as a Drinfeld $\FF_q[t]$-module.  Suppose $u_1, \dots, u_n \in
\CC_\infty$ satisfy $\exp_\rho(u_i) \in \ok$ for each $i=1, \dots,
n$ and are linearly independent over $K_\rho$.  Then the set
\[
  \cup_{j=1}^{r\ell} \bigl\{ F_{\eta_j}(u_1), \dots, F_{\eta_j}(u_n)
\bigr\}
\]
is algebraically independent over $\ok$, whence so is the
subset $\cup_{i=1}^{n} \cup_{j=1}^r \{ F_{\delta_j}(u_i)\}$.

By these same considerations, Theorem~\ref{T:Thm2Introd} is now a special case of Theorem~\ref{T:Thm3Introd}.  If we select periods $\omega_1, \dots, \omega_{r/s} \in
\Lambda_\rho$ that are linearly independent over $K_\rho$, the $r^2/s$ quantities
\[
  \cup_{j=1}^{r} \bigl\{ F_{\eta_j}(\omega_1), \dots,
F_{\eta_j}(\omega_{r/s}) \bigr\}
\]
are algebraically independent over $\ok$, and so $\trdeg_{\ok} \ok(\rP_\rho) = r^2/s$.
\end{proof}

\appendix
\section{Subgroups of a semidirect product\\ by Brian Conrad} \label{appendix}

\subsection{Main result}\label{intro}

Let $k$ be a field, $k'$ a nonzero finite reduced $k$-algebra
(i.e., $k' = \prod k'_i$ for finite extension fields $k'_i/k$),
and $G'$ a reductive group over ${\rm{Spec}}(k')$
with connected fibers.  We allow for the possibility that
$k'$ is not $k$-\'etale (i.e., some $k'_i$ is not separable over $k$).
Let $V$ be a commutative smooth connected unipotent $k$-group
equipped with a left action by the Weil restriction $G = {\rm{R}}_{k'/k}(G')$.
By \cite[Prop.\,A.5.2(4), Prop.\,A.5.9]{pred}, the affine finite type $k$-group scheme $G$
is smooth and connected.
Beware that if $k'$ is not $k$-\'etale then
$G$ is {\em not} reductive when
$G' \rightarrow {\rm{Spec}}(k')$ has a nontrivial fiber over
some point in the non-\'etale locus of $k'$ over $k$ \cite[Ex.\,1.6.1]{pred}.

For a maximal $k$-torus $T$
in $G$, the scheme-theoretic fixed locus $V^T$ is smooth and connected (since
the centralizer of $T$ in $G \ltimes V$ is smooth and connected yet equals $Z_G(T) \ltimes V^T$).  We will be interested in cases where
$V^T = 0$ for some (equivalently, all)~$T$.

\begin{remark}
The reason we do not require $k'$ to be a field in general is that in proofs
it is useful to extend the ground field $k$ to a separable closure $k_s$,  and
$k'_s := k' \otimes_k k_s$ is typically not a field.  The $k_s$-algebra $k'_s$ arises because
for an affine $k'$-scheme $X'$ of finite type we have
${\rm{R}}_{k'/k}(X')_{k_s} = {\rm{R}}_{k'_s/k_s}(X'_{k'_s})$.
\end{remark}

Here is the situation that is of most interest to us.

\begin{example}\label{exglv}
Suppose $G'$ is
equipped with a linear representation $\rho'$ on
a finitely generated $k'$-module $V'$ such that ${V'}^{Z'} = 0$ for some central torus $Z'$ in $G'$.
Also assume that the action of the Lie algebra $\mathfrak{z}'$ on $V'$ satisfies
${V'}^{\mathfrak{z}'} = 0$
(equivalently, over every factor field of $k'_s = k' \otimes_k k_s$, each weight
for the $Z'$-action on $V'$ is not divisible by ${\rm{char}}(k)$
in the geometric character lattice of the corresponding fiber of $Z'$). For example,
we could take $G' = {\rm{GL}}(\mathscr{V}')$ for a finitely generated
$k'$-module $\mathscr{V}'$ and $\rho'$ to be the standard representation of $G'$ on the
direct sum $V' = {\mathscr{V}'}^{\oplus n}$ for any $n  > 0$.

Let $V$ be the vector group over $k$ underlying $V'$,
and equip it with its natural left action by $G := {\rm{R}}_{k'/k}(G')$ (acting via Weil restriction
of $\rho'$).  The center of
$G$ contains ${\rm{R}}_{k'/k}(Z')$, so the maximal $k$-torus $Z$
in ${\rm{R}}_{k'/k}(Z')$ is contained in every maximal $k$-torus $T$ of $G$.
It is easy to check
that the weights for the action of $Z_{k_s}$ on $V_{k_s}$ are all nontrivial,
so $V^Z = 0$ and hence $V^T = 0$ for all $T$.
\end{example}

The main result of this appendix is:

\begin{proposition}\label{glvprop} With notation and hypotheses as in Example {\rm{\ref{exglv}}}, let
$\Gamma \subset G \ltimes V$ be a smooth closed
$k$-subgroup such that $\Gamma \rightarrow G$ is surjective.
The scheme-theoretic intersection $\Gamma \cap V$ is equal to $W := {\rm{R}}_{k'/k}(W')$
for a unique $k'$-submodule $W'$ of $V'$ that is moreover
$G'$-stable, and there exists a unique $v \in (V/W)(k) = V'/W'$
such that $\Gamma$ is the preimage under
$G \ltimes V \twoheadrightarrow G \ltimes (V/W)$ of the $v$-conjugate of $G$ in $G \ltimes (V/W)$.

In particular, $\Gamma \cap V$ is smooth and $\Gamma$
is connected, and if $k'$ is a field with $G'$ acting irreducibly on $V'$
over $k'$ then either
$\Gamma = G \ltimes V$ or $\Gamma$ is the $v$-conjugate of $G$ for a unique $v \in V(k)$.
\end{proposition}

\begin{remark}\label{remprop} The description of $W$ in terms of a $k'$-submodule $W'$  of $V'$
relies crucially on the hypothesis that ${V'}^{\mathfrak{z}'} = 0$ in Example \ref{exglv}.
For example, if ${\rm{char}}(k) = p > 0$ and
$k'/k$ is a purely inseparable extension of degree $p$ then consider
$V' = k'$ equipped with the action of $Z' = G' = {\rm{GL}}_1$ via $t.x = t^p x$.
Note that ${V'}^{\mathfrak{z}'} = V'$
even though ${V'}^{Z'} = 0$,
and the canonical $k$-subgroup $W = \mathbf{G}_{\rm{a}}$ in $V = {\rm{R}}_{k'/k}(\mathbf{G}_{\rm{a}})$
is $G$-stable since ${k'}^p \subseteq k$.  Clearly $\Gamma := G \ltimes W$ satisfies
$\Gamma \cap V = W$ scheme-theoretically, so $\Gamma$ violates the conclusion of
Proposition \ref{glvprop} concerning the structure of $W$.
\end{remark}

\subsection{Proof of Proposition \ref{glvprop}}\label{redsec}

In view of the uniqueness assertion for $v$ in Proposition \ref{glvprop},
by Galois descent we may and do assume
$k = k_s$.  The action of $G \ltimes V$ on the normal $k$-subgroup $V$
factors through the natural action of the quotient $G$ since $V$ is commutative.
Since $\Gamma \rightarrow G$ is surjective and $\Gamma(k)$ is Zariski-dense
in $\Gamma$ (as $k = k_s$ and $\Gamma$ is smooth), it follows
that $\Gamma(k)$ has Zariski-dense image in $G$.  Thus,
a smooth closed $k$-subgroup of $V$ is $G$-stable provided
that it is normalized by $\Gamma(k)$ inside of $G \ltimes V$.
In particular, the Zariski closure $V_0$ of $\Gamma(k) \cap V(k)$ in $V$ is
a smooth closed $k$-subgroup of $V$ that is
stable under the $G$-action on $V$.
If $\Gamma \cap V$ is going to be smooth then it must necessarily equal $V_0$,
so this motivates our work with $V_0$ in what follows.
Much later in the argument we will prove that $\Gamma \cap V = V_0$.

Any surjection between smooth connected affine
$k$-groups that is equivariant for actions by a torus
restricts to a surjection between centralizers for the torus actions,
so $V^Z \rightarrow (V/V_0)^Z$ is surjective.
Thus, $(V/V_0)^Z = 0$ since we assume $V^Z = 0$.
Before we address the structure of
$V_0$, we first verify the following lemma that amounts to the uniqueness for $v$
once we know that $\Gamma \cap V$ is smooth.

\begin{lemma} There is at most one $v \in (V/V_0)(k)$ such that
$\Gamma/V_0$ is the $v$-conjugate of $G$ inside of $G \ltimes (V/V_0)$.
\end{lemma}

\begin{proof}
For any $k$-algebra $A$ and $A$-valued points $g, h$ of $G$
and $v, w$ of $V/V_0$, we have
\begin{equation}\label{conj}
(g,v)(h,w)(g,v)^{-1} = (gh, h^{-1}v+w)(g^{-1},-gv) = (ghg^{-1}, gh^{-1}v - gv + gw).
\end{equation}
Setting $g = 1$, for any $v \in (V/V_0)(k)$ we see that the $v$-conjugate of
$(h,w)$ is $(h, h^{-1}v - v + w)$.  Thus, the $v$-conjugate of $G$ in $G \ltimes (V/V_0)$
is the graph of the map $G \rightarrow V/V_0$ defined by
$h \mapsto h^{-1}v - v$. Uniqueness of $v$ reduces to the property
$(V/V_0)(k)^G = 0$, and this holds because $(V/V_0)^Z = 0$.
\end{proof}

Suppose we could show that the image $\overline{\Gamma}$ of $\Gamma$
in $G \ltimes (V/V_0)$ contains a $(V/V_0)(k)$-conjugate of $G$.
Since $V(k) \rightarrow (V/V_0)(k)$ is surjective (as $k = k_s$ and $V_0$ is smooth),
such a containment would bring us to the situation (after a $V(k)$-conjugation on $\Gamma$)
that $\Gamma = G \ltimes (\Gamma \cap V)$,
which forces the scheme-theoretic intersection $\Gamma \cap V$
to be smooth (as a direct factor of a smooth $k$-scheme is smooth)
and therefore equal to $V_0$,
so we would be done.  Our problem is therefore
reduced to establishing that $V_0 = {\rm{R}}_{k'/k}(W')$
for a $G'$-stable $k'$-submodule $W'$ in $V'$ (such a $W'$ is clearly unique
as a $k'$-submodule, even without reference to the $G'$-stability condition)
and that $\overline{\Gamma}$ contains a $(V/V_0)(k)$-conjugate
of $G$ (or equivalently that
some $V(k)$-conjugate of $\Gamma$ in $G \ltimes V$ contains
$G$).

Let $T \subseteq \Gamma$ be a maximal $k$-torus, so its image in $G$ is a maximal
$k$-torus (as $\Gamma \rightarrow G$ is surjective).  The maximal tori in
$G \ltimes V$ are conjugates of maximal tori of $G$
since $V$ is unipotent, so by dimension reasons it
follows that $T$ is maximal as a $k$-torus of $G \ltimes V$.
It is harmless to replace $\Gamma$ with a $(G \ltimes V)(k)$-conjugate
since we aim to show that some $V(k)$-conjugate of $\Gamma$ contains $G$,
and by a result of Grothendieck
all maximal $k$-tori in $G \ltimes V$ are $k$-rationally conjugate since $k = k_s$
(see \cite[Prop.\,A.2.10]{pred} for a self-contained elementary proof).
 Thus, we may arrange that $T$ coincides with
any desired maximal $k$-torus of $G = {\rm{R}}_{k'/k}(G')$ inside of
$G \ltimes V$.  Fix such a $T$.  Note that now $T$ contains $Z$,
as all maximal $k$-tori in $G$ must contain any central $k$-torus (such as $Z$).

The quotient map $\Gamma \twoheadrightarrow G$ must carry
the Cartan $k$-subgroup $Z_{\Gamma}(T)$ onto $Z_G(T)$.
By \cite[Prop.\,A.5.15]{pred},
there is a unique maximal $k'$-torus $T'$ in $G'$ such that $T \subseteq
{\rm{R}}_{k'/k}(T')$, and moreover $Z_G(T) = Z_G({\rm{R}}_{k'/k}(T')) = {\rm{R}}_{k'/k}(T')$
(the final equality by comparison of $k$-points, as $Z_G(T)$ is smooth).
Thus, $\dim Z_{\Gamma}(T) \ge \dim {\rm{R}}_{k'/k}(T')$.  But
$$Z_{\Gamma}(T) \subseteq Z_{G \ltimes V}(T) = Z_G(T)$$
since $V^{T} = 0$ (as $V^Z = 0$ and $Z \subseteq T$).  Hence, $Z_{\Gamma}(T) = {\rm{R}}_{k'/k}(T')$ inside of
$G \ltimes V$.    In particular, ${\rm{R}}_{k'/k}(T') \subseteq \Gamma$, so
${\rm{R}}_{k'/k}(Z') \subseteq \Gamma$.  Hence,
the $k$-subgroup $V_0$ in $V$ is an ${\rm{R}}_{k'/k}(Z')$-stable closed $k$-subgroup
of $V$.  The structure of $V_0$ is therefore determined by:

\begin{lemma}\label{w}
Let $k$ be a field,
$k'$ a nonzero finite reduced $k$-algebra,
and $V'$ a finitely generated $k'$-module equipped
with a linear action by a $k'$-torus $Z'$ such that
the action of the Lie algebra $\mathfrak{z}' = {\rm{Lie}}(Z')$ on $V'$ satisfies
${V'}^{\mathfrak{z}'} = 0$.
Any ${\rm{R}}_{k'/k}(Z')$-stable smooth closed $k$-subgroup $H$ of
$V := {\rm{R}}_{k'/k}(V')$ has the form ${\rm{R}}_{k'/k}(W')$ for a unique $k'$-submodule $W'$ of $V'$.
\end{lemma}

As in Remark \ref{remprop},
the hypothesis on the $\mathfrak{z}'$-action cannot be dropped.
The proof of Lemma \ref{w} is rather long, so let us first see how to use it to
complete the proof of Proposition \ref{glvprop}.

We write $k' = \prod k'_i$ for fields $k'_i$,
and correspondingly $V' = \prod V'_i$ for a vector space $V'_i$ over $k'_i$.
Let $G'_i$ denote the $k'_i$-fiber of $G'$ (so $G = \prod {\rm{R}}_{k'_i/k}(G'_i)$).
The $k'$-torus $T'$ contains the central torus $Z'$, and we recall that ${V'}^{Z'} = 0$.
Granting Lemma \ref{w}, the Zariski closure $V_0$ of $(\Gamma \cap V)(k)$ in $V$
is ${\rm{R}}_{k'/k}(W')$ for
a unique $k'$-submodule $W'$ of $V'$.

Note that $V_0$ is $G$-stable,
due to the surjectivity of $\Gamma \rightarrow G$,
so consideration of $k$-points shows that $W'$ is a $G'$-stable $k'$-submodule of $V'$
(as $k = k_s$).
Since $V_0(k) = (\Gamma \cap V)(k)$,
we see that $\Gamma(k) \cap V(k)$ coincides with the $k'$-submodule $W'$ of $V'$.
The $k$-group $\Gamma$ is therefore the full preimage in $G \ltimes V$ of
a smooth closed $k$-subgroup $\overline{\Gamma}$ in $G \ltimes (V/V_0)$
such that $\overline{\Gamma}(k) \cap (V/V_0)(k) = 0$.
Let $\overline{V} = V/V_0$.

The maximal central torus $Z$ of ${\rm{R}}_{k'/k}(Z')$
satisfies $V^Z = 0$.  Hence, $\overline{V}^Z = 0$ as well,
since $Z$ is a torus.  For any $\gamma = (g,\overline{v}) \in \overline{\Gamma}(k)$
the effect of $\gamma^{-1}$-conjugation on an element
$z \in Z(k)$ is the same as conjugation by $\overline{v}^{-1} = -\overline{v}$.
Thus, $\gamma^{-1} z \gamma = \overline{v}^{-1}z\overline{v} = (z, \overline{v} - z\overline{v})$, so
since $Z(k) \subseteq \overline{\Gamma}(k)$ we have $z\overline{v}-\overline{v} \in
\overline{\Gamma}(k) \cap
\overline{V}(k) = 0$.
Letting $z$ vary, it follows that $\overline{v} \in \overline{V}^Z = 0$.
In other words, $\gamma \in G(k)$ inside of $(G \ltimes V)(k)$.
That is, $\overline{\Gamma} \subseteq G$ inside of $G \ltimes V$.  But
the projection $\overline{\Gamma} \rightarrow G$ is surjective, so
$\overline{\Gamma} = G$ inside of $G \ltimes \overline{V}$ as required.
This completes the proof of Proposition \ref{glvprop}, conditional on Lemma \ref{w}.

\begin{proof}[Proof of Lemma {\rm{\ref{w}}}]
By Galois descent (in view of the uniqueness claim), we may and do assume $k = k_s$.
The uniqueness of $W'$ is clear, so the problem is its existence.
Let $\{k'_i\}$ be the set of factor fields of $k'$, and $Z'_i$ the $k'_i$-fiber of $Z'$.
Consider the $k$-group decomposition ${\rm{R}}_{k'/k}(Z') = \prod {\rm{R}}_{k'_i/k}(Z'_i)$,
so the maximal $k$-torus $Z$ of ${\rm{R}}_{k'/k}(Z')$ has the form $\prod Z_i$
for the maximal $k$-torus $Z_i$ in ${\rm{R}}_{k'_i/k}(Z'_i)$.
Note that the character groups of $Z'_i$ and $Z_i$ naturally coincide
since $k'_i/k$ is purely inseparable (and $k = k_s$).
We may replace $H$ with $H^0$, so $H$ is connected.
Our goal is to show that $H$ contains the image of its projection into each
${\rm{R}}_{k'_i/k}(V'_i)$.

Choose an index $i_0$ and let $S = \prod_{i \ne i_0} {\rm{R}}_{k'_i/k}(Z'_i)$,
so $V^S = {\rm{R}}_{k'_{i_0}/k}(V'_{i_0})$ where $V'_i$ is the $k'_i$-factor of $V'$.
The centralizer $H^S$ for the $S$-action on $H$ is smooth and connected (since $S$ is a torus
and $H$ is smooth and connected), and the image of
$H^S$ in ${\rm{R}}_{k'_{i_0}/k}(V'_{i_0})$ is the same
as the image of $H$ since the formation of $S$-invariants commutes
with the formation of images under homomorphisms between
smooth connected affine $k$-groups.  Thus, we can replace $(k'/k, V', Z',H)$ with $(k'_{i_0}/k, V'_{i_0}, Z'_{i_0},H^S)$ to reduce to the case that $k'$ is a field.

With $k'$ now arranged to be a field, we will show that $H$
contains the image of its projection into each ${\rm{R}}_{k'/k}(V'_{\chi'})$,
with $V'_{\chi'}$ varying through the $Z'$-weight spaces of $V'$.
Pick $\chi'_0$ such that $V'_{\chi'_0} \ne 0$, and let $T' = (\ker \chi'_0)_{\rm{red}}^0$ be
the codimension-1 torus  in $Z'$ killed by $\chi'_0$.  Let $T$ be the maximal $k$-torus in
${\rm{R}}_{k'/k}(T')$, so ${\rm{X}}(T)$ is naturally identified with ${\rm{X}}(T')$.
Clearly $V^T = {\rm{R}}_{k'/k}({V'}^{T'})$, and this is the direct product of the factors ${\rm{R}}_{k'/k}(V'_{\psi'})$
where $\psi' \in {\rm{X}}(Z'/T') \simeq \mathbf{Z}$ varies through those $\chi'$ that are
rational multiples of $\chi'_0$.  Also, the
$T$-centralizer
$H^T$ is smooth and connected (as for a torus action on any smooth
connected affine group), and $H^T \rightarrow {\rm{R}}_{k'/k}(V'_{\chi'_{0}})$ has
the same image as $H$ (since $T$ acts trivially on ${\rm{R}}_{k'/k}(V'_{\chi'_{0}})$).
Since ${\rm{X}}(Z'/T')$ is saturated
in ${\rm{X}}(Z')$, we may replace $(V', Z', H)$ with $({V'}^{T'}, Z'/T', H^T)$
to reduce to the case that $Z' \simeq {\rm{GL}}_1$.

The image of $H$ in each ${\rm{R}}_{k'/k}(V'_{\chi'})$ has group of $k$-points
that is an additive subgroup $W'_{\chi'}$ of $V'_{\chi'}$ stable under
the action of ${\rm{R}}_{k'/k}(Z')(k) = {k'}^{\times}$ through
$\chi'$.  That is, writing $\chi'(t) = t^m$ with ${\rm{char}}(k)\nmid m$
(due to our hypothesis on the vanishing of the $\mathfrak{z}'$-invariants),
the surjectivity of the $m$th-power endomorphism of ${k'}^{\times}$
implies that $W'_{\chi'}$ is a $k'$-linear subspace of $V'_{\chi'}$
for every $\chi'$.  We can replace $V'_{\chi'}$ with $W'_{\chi'}$
for every $\chi'$ to arrange that $H \rightarrow {\rm{R}}_{k'/k}(V'_{\chi'})$ is surjective
for all $\chi'$.  Now we aim to prove $H = V$;
i.e., $H$ contains ${\rm{R}}_{k'/k}(V'_{\chi'})$ for all $\chi'$.

Fix a choice of $\chi'_0$. The idea is to find a functorial way
of selecting a $Z$-stable smooth closed $k$-subgroup $F(G)$ of any
smooth affine $k$-group $G$ equipped with a $Z$-action so that the following formal properties
hold:  (i) $F(V) = {\rm{R}}_{k'/k}(V'_{\chi'_0})$
(avoiding any reference to the linear structure
on $V'$ over $k'$ or the linear structure on $V$ over $k$!), (ii) $F$ is a ``projector'' in
the sense $F(F(G)) = F(G)$ for any $G$, (iii) $F$ carries surjections to surjections
(without smoothness hypotheses on these surjections).
Once such an $F$ is in hand, applying it to the $Z$-equivariant
quotient map $H \twoheadrightarrow {\rm{R}}_{k'/k}(V'_{\chi'_0}) = F(V)$
yields a surjective map $F(H) \twoheadrightarrow F(F(V)) = F(V)$, yet
$H$ is a $Z$-stable closed $k$-subgroup of $V$, so by functoriality of
the subgroup assignment $F$ we see that
$F(H)$ is a $k$-subgroup of $F(V)$ inside of $V$.  In other words,
$F(H)$ is a $k$-subgroup of $F(V) = {\rm{R}}_{k'/k}(V')$ inside of $V$ such that the
natural projection $V \twoheadrightarrow {\rm{R}}_{k'/k}(V'_{\chi'_0}) = F(V)$ extending
the identity on $F(V)$ restricts to a surjection on $F(H)$, forcing $F(H) = F(V)$.
This implies that $H$ contains $F(H) = F(V) = {\rm{R}}_{k'/k}(V'_{\chi'_0})$ inside of $V$,
as desired.

We will not quite find such a functor $F$ in the generality just described. Instead,
we will first carry out some preliminary reduction steps to acquire finer properties
for $\chi'_0$ (in comparison with all other $Z'$-weights on $V'$), and then we will construct such an $F$.
Our aim is to prove that $H$ contains ${\rm{R}}_{k'/k}(V'_{\chi'_0})$
for an arbitrary but fixed choice of $\chi'_0$, so we begin by
composing the isomorphism $Z' \simeq {\rm{GL}}_1$ with inversion
if necessary to ensure that $\chi'_{0}(t) = t^n$ with $n > 0$.
Proceeding by descending induction (or by contradiction),
we may and do assume that $H$ contains ${\rm{R}}_{k'/k}(V'_{\chi'})$
for any $\chi':t \mapsto t^m$ with $m > n$.  It is clearly harmless to pass to the quotient of
$V'$ by the span of such $V'_{\chi'}$ with ``weight'' larger than $n$
(and pass to the quotient of $H$ by the $k$-subgroup directly spanned by the
${\rm{R}}_{k'/k}(V'_{\chi'})$ for such $\chi'$), so now every positive weight $\chi'$ that
occurs in $V'$ has the form $\chi'(t) = t^m$
with $m \le n$.

The maximal
$k$-torus $Z$ in ${\rm{R}}_{k'/k}(Z') = {\rm{R}}_{k'/k}({\rm{GL}}_1)$ is the evident
${\rm{GL}}_1$.  Since the $k'$-group $Z'[n] = \mu_n$ is finite \'etale (as ${\rm{char}}(k) \nmid n$)
and $k'/k$ is purely inseparable, clearly ${\rm{R}}_{k'/k}(Z'[n]) = Z[n]$
and $V^{Z[n]} = {\rm{R}}_{k'/k}({V'}^{Z'[n]})$. Thus, exactness of the formation of $\mu_n$-invariants allows us to
replace $Z'$ with $Z'/Z'[n]$, replace
$V'$ with ${V'}^{Z'[n]}$ (which leaves $V'_{\chi'_0}$ unchanged),
and replace $H$ with $(H^{\mu_n})^0$.
This brings us to the case that $\chi'_0(t) = t$ and all other weights are negative.
In other words, each $\chi' \ne {\chi'_0}$ satisfying $V'_{\chi'} \ne 0$
necessarily has the form $\chi' = {\chi'_0}^m$ for some $m < 0$
(depending on $\chi'$).

Finally, we bring in a procedure that compatibly separates positive weights from negative weights
in both $H$ and ${\rm{R}}_{k'/k}(V')$ under the action of $Z = {\rm{GL}}_1$.
The basic construction we need is systematically
developed in \cite[2.1]{pred} in a functorial manner, and it goes as follows.
By \cite[Rem.\,2.1.11]{pred}, for any affine $k$-group scheme $G$ of finite type
and left action $\mu$ of ${\rm{GL}}_1$ on $G$, there is a closed $k$-subgroup scheme
$U_G(\mu)$ of $G$ representing the functor that
assigns to any $k$-algebra $R$ the subgroup of points $g \in G(R)$ such that
the $R$-scheme morphism ${\rm{GL}}_1 \rightarrow G_R$ defined by $t \mapsto t.g$
extends (necessarily uniquely) to an $R$-scheme morphism $\mathbf{A}^1_R \rightarrow G_R$.
More specifically, by \cite[Rem.\,2.1.11]{pred},
$U_G(\mu)$ is  smooth and connected when $G$ is smooth,
and if $G \rightarrow \overline{G}$ is a ${\rm{GL}}_1$-equivariant flat surjection
between connected affine $k$-groups of finite type
then $U_G(\mu) \rightarrow U_{\overline{G}}(\mu)$ is a flat quotient map.
Note also that if $G'$ is a closed $k$-subgroup of $G$ that is
stable under the action of ${\rm{GL}}_1$ then $U_{G'}(\mu) = G' \cap U_G(\mu)$.
The functor $(G,\mu) \rightsquigarrow U_G(\mu)$ will play the role of the functor $F$ in the
formal considerations given above.

Define the isomorphism $\lambda:
{\rm{GL}}_1 \simeq Z$ to be the inverse of the preferred isomorphism used above
(i.e., the inverse of the restriction of
${\rm{R}}_{k'/k}(\chi'_0):{\rm{R}}_{k'/k}(Z') \simeq {\rm{R}}_{k'/k}({\rm{GL}}_1)$
to maximal $k$-tori), and use this to identify the natural $Z$-action on $V$ with a
${\rm{GL}}_1$-action. (This is exactly the natural scaling action
arising from the evident linear structure on $V = {\rm{R}}_{k'/k}(V')$ over
$k$, in view of how we have changed $Z'$ in relation to $\chi'_0$.)  In this way we get smooth connected
$k$-subgroups $U_H(\lambda)$ in $H$ and $U_V(\lambda)$ in $V$,
with $U_H(\lambda) = H \cap U_V(\lambda)$ scheme-theoretically.
Every $\chi' \ne \chi'_0$ such that $V'_{\chi'} \ne 0$
necessarily has the form ${\chi'_0}^m$ with $m < 0$ (depending on $\chi'$), so
$U_V(\lambda) = {\rm{R}}_{k'/k}(V'_{\chi'_0})$.
In particular, the $k$-subgroup $\overline{H} := {\rm{R}}_{k'/k}(V'_{\chi'_0})$ in
$V$ is stable under the ${\rm{GL}}_1$-action through $\lambda$
and satisfies $U_{\overline{H}}(\lambda) = \overline{H}$.

Since we reduced to the case that the composite map
$H \hookrightarrow V \twoheadrightarrow {\rm{R}}_{k'/k}(V'_{\chi'})$ is surjective
for all $\chi'$, by taking $\chi' = \chi'_0$ we deduce that the ${\rm{R}}_{k'/k}(Z')$-equivariant map
$H \rightarrow \overline{H}$
is a ${\rm{GL}}_1$-equivariant flat quotient map.
Thus, the induced map $U_H(\lambda) \rightarrow U_{\overline{H}}(\lambda) = \overline{H}$
is surjective.  In other words, the surjective composite map
$H \hookrightarrow {\rm{R}}_{k'/k}(V') \twoheadrightarrow {\rm{R}}_{k'/k}(V'_{\chi'_0})$
restricts to a surjective map on the $k$-subgroup $U_H(\lambda)$.
But $U_H(\lambda)$ viewed inside
of $V$ is a $k$-subgroup of $U_V(\lambda) = {\rm{R}}_{k'/k}(V'_{\chi'_0})$,
so we conclude that  the inclusion  $U_H(\lambda) \hookrightarrow
{\rm{R}}_{k'/k}(V'_{\chi'_0})$ is an equality (since the natural projection
$V \twoheadrightarrow {\rm{R}}_{k'/k}(V'_{\chi'_0})$ restricts
to the identity on the $k$-subgroup ${\rm{R}}_{k'/k}(V'_{\chi'_0})$).   This
shows that the $k$-subgroup $H$ in $V$ contains
${\rm{R}}_{k'/k}(V'_{\chi'_0})$, as required.
\end{proof}

\subsection{A generalization}

In the general setting
introduced before Example \ref{exglv} there is no ``linear structure''
imposed on $V$ or the $G$-action on $V$.  In that generality we have the following
replacement for Proposition \ref{glvprop} away from characteristic~2:

\begin{theorem}\label{mainthm}
Using notation and hypotheses as at the start of \S{\rm{\ref{intro}}},
assume ${\rm{char}}(k) \ne 2$ and $V^T = 0$ for some maximal
$k$-torus $T$ of $G$. Let $\Gamma$ be a smooth closed $k$-subgroup of $G \ltimes V$
such that the composite map $\Gamma \rightarrow G$ is surjective.

The scheme-theoretic intersection $V_0 := \Gamma \cap V$ is smooth
and $G$-stable,
and there is a unique $v \in (V/V_0)(k)$ such that
$\Gamma$ is the preimage under
$G \ltimes V \twoheadrightarrow G \ltimes (V/V_0)$ of
the $v$-conjugate of $G$ in $G \ltimes (V/V_0)$.
\end{theorem}

The proof of Theorem \ref{mainthm} will appear elsewhere, as it
involves techniques of an entirely different nature from the proof of Proposition \ref{glvprop}:
the structure theory of pseudo-reductive groups developed in \cite{pred}
(even though the statement of Theorem \ref{mainthm} does not involve pseudo-reductive groups).

Theorem \ref{mainthm} is probably valid
in characteristic~2 under additional restrictions,
including that the maximal geometric semisimple quotient $G_{\overline{k}}^{\rm{ss}}$
has no simple factor of type C.
The necessity of additional hypotheses related to type C in a characteristic-2 version of Theorem \ref{mainthm}
is not merely a matter of technique, but rather is due to explicit counterexamples.  Such counterexamples
illuminate the meaning of Theorem \ref{mainthm} and help one to appreciate
the good fortune of the characteristic-free nature of Proposition \ref{glvprop}, so we now
provide a family of counterexamples to a characteristic-2 version of Theorem \ref{mainthm}
without type-C restrictions.

\begin{example}\label{noisom}
Let $k$ be a field of characteristic~2, and
$(V,q)$ a non-degenerate quadratic space of rank $2n+1$ over $k$.
The associated symmetric  bilinear form $B_q$ on $V$ is alternating
with $V^{\perp}$ of dimension 1, so there is an induced non-degenerate alternating
form $\overline{B}_q$ on the quotient $\overline{V} = V/V^{\perp}$ of dimension $2n$.
Thus, we get a map ${\rm{SO}}(q) \rightarrow {\rm{Sp}}(\overline{B}_q)$.

For example, if $q = x_0^2 + x_1 x_2 + \dots + x_{2n-1} x_{2n}$ is
the standard split quadratic form on $V = k^{2n+1}$ then
$V^{\perp} = k e_0$ and the map
\begin{equation}\label{somap}
{\rm{SO}}_{2n+1} = {\rm{SO}}(q) \rightarrow {\rm{Sp}}(\overline{B}_q) = {\rm{Sp}}_{2n}
\end{equation}
is projection onto the lower-right $2n \times 2n$ block matrix. This map is surjective
and its kernel consists
of lower-triangular unipotent matrices whose entries below the diagonal all vanish
apart from $\alpha_2$ entries along the left column (below the upper-left entry of 1).
As a group scheme, this kernel is $\alpha_2^{2n}$ with components corresponding to
matrix entries along the left column,
and the induced action on it by ${\rm{Sp}}_{2n}$ is via
the identification of $\alpha_2^{2n}$ as the Frobenius kernel in the standard
representation space of rank $2n$.
(To put this in perspective, we note that by
 \cite[Lemma 2.2]{pyu},
over an algebraically closed field the only isogenies
between absolutely simple and connected semisimple groups
for which the kernel is nontrivial and unipotent
are the isogenies (\ref{somap}) in characteristic
2, up to an isomorphism on the source and target.)

In general, let $W$ denote the vector space ${\rm{Hom}}(V^{\perp}, \overline{V})$
corresponding to ``left column below the top entry'' in $\mathfrak{so}_{2n+1} \subset \mathfrak{gl}_{2n+1}$,
and equip it with the standard action by ${\rm{Sp}}(\overline{B}_q) \subset {\rm{GL}}(\overline{V})$.
Viewing $W$ as a vector group over
$k$, let $F_{W/k}:W \rightarrow W^{(2)}$ denote its relative Frobenius isogeny.
There is an exact sequence
$$1 \rightarrow \ker F_{W/k} \rightarrow {\rm{SO}}(q) \rightarrow
{\rm{Sp}}(\overline{B}_q) \rightarrow 1$$
in which the induced left action of the quotient term on the commutative kernel is the natural one
arising from the ${\rm{Sp}}(\overline{B}_q)$-action on $W$.
The pushout along the
${\rm{Sp}}(\overline{B}_q)$-equivariant inclusion $\ker F_{W/k} \hookrightarrow W$
is an exact sequence
\begin{equation}\label{vseq}
1 \rightarrow W \rightarrow E \rightarrow {\rm{Sp}}(\overline{B}_q) \rightarrow 1
\end{equation}
in which $E$ contains the subgroup $\Gamma := {\rm{SO}}(q)$ mapping
via a degree-$2^{2n}$ infinitesimal isogeny onto ${\rm{Sp}}(\overline{B}_q)$
and the induced left action of ${\rm{Sp}}(\overline{B}_q) \subset {\rm{GL}}(\overline{V})$ on
the commutative kernel $W = {\rm{Hom}}(V^{\perp}, \overline{V})$ is the natural one.
(By choosing a basis of the line $V^{\perp}$, we get an
${\rm{Sp}}(\overline{B}_q)$-equivariant isomorphism $W \simeq \overline{V}$.)
Thus, for any maximal torus $T$ in ${\rm{Sp}}(\overline{B}_q)$
the subspace $W^T$ vanishes because the center $\mu_2$ of the symplectic group
acts by ordinary scaling (and so even $W^{\mu_2}$ vanishes).
We claim that (\ref{vseq}) splits as a semidirect product, so inside of
${\rm{Sp}}(\overline{B}_q) \ltimes \overline{V}$ we get
a counterexample to the conclusion of Theorem \ref{mainthm}
in type-${\rm{C}}$ cases in characteristic~2.

To construct a splitting of (\ref{vseq}), more generally consider an arbitrary short exact sequence
\begin{equation}\label{esplit}
1 \rightarrow U \rightarrow E \rightarrow {\rm{Sp}}_{2n} \rightarrow 1
\end{equation}
of linear algebraic groups over a field $k$, where $U \simeq \mathbf{G}_{\rm{a}}^N$ is a vector group.
Assume that the induced action of the center $\mu = \mu_2$ of ${\rm{Sp}}_{2n}$ on $U$
is the natural scaling action. We shall prove that (\ref{esplit}) is split.

Note
that $E$ is necessarily a smooth connected affine $k$-group.
The exact sequence (\ref{esplit})
pulls back to an extension of $\mu$ by $U$ lifting
the natural linear action of $\mu$ on $\mathbf{G}_{\rm{a}}^N$.  We claim that
this pullback splits, so $\mu$ lifts
into $E$.   To prove that the pointed set ${\rm{Ex}}_k(\mu_2, \mathbf{G}_{\rm{a}}^N)$
of such extensions is trivial, we immediately reduce to the case $N = 1$.
The case ${\rm{char}}(k) \ne 2$ is easy (since then
$\mu_2 = \mathbf{Z}/2\mathbf{Z}$ and doubling is an automorphism of
$\mathbf{G}_{\rm{a}}$).  Suppose instead
that ${\rm{char}}(k) = 2$.  Then consideration of Frobenius kernels
reduces the problem to the vanishing of ${\rm{Ex}}_k(\mu_2,\alpha_2)$ over
fields of characteristic $p = 2$ (relative to the usual scaling action of $\mu_2$ on $\alpha_2$),
and this in turn follows from an easy calculation with $p$-Lie algebras (and works
just as well with $\mu_p$ and $\alpha_p$ for $p > 2$).

For $k$-subgroup $\mu' \subset E$ isomorphically lifting the center $\mu$ of
${\rm{Sp}}_{2n}$, consider the scheme-theoretic centralizer $E^{\mu'}$
of $\mu'$ in $E$.  By
\cite[Prop.\,A.8.10]{pred}, $E^{\mu'}$ is smooth and ${\rm{Lie}}(E^{\mu'})$
is equal to ${\rm{Lie}}(E)^{\mu'}$.  By exactness of $\mu'$-invariants on linear representations,
the surjective map ${\rm{Lie}}(E) \rightarrow \mathfrak{sp}_{2n}$ (with kernel
${\rm{Lie}}(U)$) induces a surjective map
${\rm{Lie}}(E)^{\mu'} \rightarrow \mathfrak{sp}_{2n}$. This latter surjection is an isomorphism since
its kernel is ${\rm{Lie}}(U)^{\mu'} = {\rm{Lie}}(\mathbf{G}_{\rm{a}}^N)^{\mu} =
({\rm{Lie}}(\mathbf{G}_{\rm{a}})^{\mu})^N = 0$.  Hence,
$(E^{\mu'})^0 \rightarrow {\rm{Sp}}_{2n}$ is an isogeny with \'etale kernel.
But the connected semisimple group ${\rm{Sp}}_{2n}$ is simply connected, so
$(E^{\mu'})^0 \rightarrow {\rm{Sp}}_{2n}$ is an isomorphism; this is the required splitting.
\end{example}

\bibliographystyle{amsplain}

\end{document}